\documentclass{amsart}
\usepackage[english]{babel}
\usepackage{mathrsfs}
\usepackage{amsmath}
\usepackage{amsthm}
\usepackage{amssymb}
\usepackage{indentfirst}
\usepackage[all]{xy}
\usepackage{tikz}
\usepackage{tikz-cd}
\usetikzlibrary{shapes.geometric, calc,matrix,arrows,decorations.pathmorphing,arrows.meta}
\usepackage{graphicx}
\usepackage{graphics}
\usepackage{caption}
\usepackage{extarrows}
\usepackage[enableskew, vcentermath]{youngtab}
\usepackage[utf8x]{inputenc}

\theoremstyle{plain} 
\newtheorem{prop}{Proposition}[section]
\newtheorem{thm}[prop]{Theorem}
\newtheorem{lem}[prop]{Lemma} 
\newtheorem{cor}[prop]{Corollary}

\newtheorem*{conj*}{Conjecture}
\newtheorem*{thm*}{Theorem}

\theoremstyle{remark}

\newtheorem{oss}[prop]{Remark}
\newtheorem{ex}[prop]{Example}
\newtheorem{cla}[prop]{Claim}

\newtheorem*{ques*}{Question}

\theoremstyle{definition}
\newtheorem{defn}[prop]{Definition}

\newcommand{\Or}{\operatorname{O}}

\newcommand{\CC}{\mathbb{C}}
\newcommand{\PP}{\mathbb{P}}
\newcommand{\QQ}{\mathbb{Q}}
\newcommand{\RR}{\mathbb{R}}
\newcommand{\ZZ}{\mathbb{Z}}

\newcommand{\cO}{\mathcal{O}}
\newcommand{\cJ}{\mathcal{J}}

\newcommand{\oH}{\operatorname{H}}
\newcommand{\Mon}{\operatorname{Mon}}
\newcommand{\Pic}{\operatorname{Pic}}
\newcommand{\Hilb}{\operatorname{Hilb}}
\newcommand{\IJac}{\operatorname{IJ}}

\renewcommand{\div}{\mathrm{div}}

\begin{document}
\title[Birational geometry of OG10]{Birational geometry of irreducible holomorphic symplectic tenfolds of O'Grady type}
\author{Giovanni Mongardi}
\address{Dipartimento di Matematica, Universit\'{a} degli studi di Bologna,  Piazza di porta san Donato 5, Bologna, Italia 40126  }
\email{giovanni.mongardi2@unibo.it}
\author{Claudio Onorati}
\address{Department of Mathematics, University of Oslo, PO Box 1053, Blindern, 0316 Oslo, Norway}
\email{claudion@math.uio.no}

\begin{abstract}
In this paper, we analyse the birational geometry of O'Grady ten dimensional manifolds, giving a characterization of K\"ahler classes and lagrangian fibrations. Moreover, we study symplectic compactifications of intermediate jacobian fibrations of smooth cubic fourfolds.
\end{abstract}

\maketitle

\tableofcontents


\section*{Introduction}

An irreducible holomorphic symplectic manifold is a simply connected compact K\"ahler manifold with a unique nowhere degenerate symplectic holomorphic $2$-form. Up to deformation, only few examples are known in each dimension, and the question whether they are all or not is widely open.
In dimension $2$ though there are only K3 surfaces. In this paper we concentrate on one specific deformation type, the so called OG10 type, and investigate aspects of its birational geometry. 
These manifolds are deformation equivalent to the symplectic resolution of singularities of a singular moduli space of sheaves on a K3 surface constructed by O'Grady (\cite{O'Grady:OG10}).

The geometry of an irreducible holomorphic symplectic manifold $X$ is encoded in the second integral cohomology group $\oH^2(X,\ZZ)$. Recall that such a group is torsion free and has a symmetric non-degenerate bilinear form on it, called the Beauville--Bogomolov--Fujiki form and denoted by $q_X$. The pair $(\oH^2(X,\ZZ),q_X)$ is a lattice. 

We start our investigation by studying lagrangian fibrations. If $f\colon X\to B$ is a lagrangian fibration, then the divisor $f^*\cO_B(1)$ is isotropic (i.e.\ $q_X(f^*\cO_B(1)=0$) and nef. Our first result establishes a birational converse of this fact.
\begin{thm*}[Theorem~\ref{thm:lagr}]
Let $X$ be an irreducible holomorphic symplectic manifold of OG10 type and let $\mathcal{O}(D)\in Pic(X)$ be a non-trivial isotropic line bundle.
Assume that the class $[D]$ of $\mathcal{O}(D)$ belongs to the boundary of the birational K\"ahler cone of $X$. 

Then, there exists a smooth irreducible holomorphic symplectic manifold $Y$ and a bimeromorphic map 
$\psi\colon Y\dashrightarrow X$ such that $\mathcal{O}(D)$ induces  a lagrangian fibration $p\colon Y\rightarrow \mathbb{P}^5$. 
\end{thm*}
Recall that the birational K\"ahler cone of $X$ is the union of the K\"ahler cones of all the smooth birational models of $X$.

As a straightforward corollary we get a proof of the weak splitting property conjectured by Beauville (\cite{beau07}).
\begin{thm*}[Corollary~\ref{cor:weak splitting}]
Let $X$ be a projective irreducible holomorphic symplectic manifold of OG10 type and let $D$ be an isotropic divisor on it. Let $DCH(X)\subset CH_{\mathbb{Q}}(X)$ be the subalgebra generated by divisor classes. Then the restriction of the cycle class map 
$cl_{|DCH(X)}\colon DCH(X)\to \oH^*(X,\mathbb{Q})$ is injective .
\end{thm*}

The next result describes, in a lattice-theoretic way, the ample and movable cones. In the following $\mathcal{C}(X)\subset\oH^{1,1}(X,\RR)$ is the connected component of the cone of positive classes that contains the K\"ahler cone.
\begin{thm*}[Theorem~\ref{thm:movable}, Theorem~\ref{thm:ample}]
Let $X$ be an irreducible holomorphic symplectic manifold of OG10 type. The movable cone of $X$ is an open set inside one of the connected components of $$\mathcal{C}(X)\setminus \bigcup_{D^2=-2 \text{ or } (D^2=-6 \text{ and } div(D)=3)} D^\perp.$$
The K\"ahler cone of $X$ is one of the connected components of $$\mathcal{C}(X)\setminus \bigcup_{(0>D^2\geq -4) \text{ or } (\div(D)=3 \text{ and } 0>D^2\geq -24)} D^\perp.$$
\end{thm*}
The most important tool to achieve this result is the classification of prime exceptional divisors (see Proposition~\ref{prop:pex}) and wall divisors (see Proposition~\ref{prop:wall}). Recall that prime exceptional divisors are irreducible and reduced divisors with negative square (\cite{mar_prime}): it is known (see \cite{mark_tor}) that the movable cone is contained in a prime exceptional chamber, that is in a chamber determined by prime exceptional divisors. Wall divisors are divisors with negative square and such that their orthogonal complements do not contain K\"ahler classes of any smooth birational model (\cite{Mongardi:Cones}): the ample cone is a chamber in the wall-and-chamber decomposition they produce (\cite{Mongardi:Cones}). 

This classification is obtained in two ways. First we construct in an explicit and geometric way examples of wall divisors, studying the birational transformations associated. Then we prove that they are the only possibilities, by using recent results about the minimal model program. This last step is the most technical part of the argument. We use a result of Ch.\ Lehn and Pacienza on the minimal model program for irreducible holomorphic symplectic manifolds (\cite{lp}) to reduce the problem to a singular moduli space of sheaves on a K3 surface. Here we can apply the result of Meachan and Zhang (\cite{Meach_Zhang}) to show that the divisors we found are all.

Finally, as an application of these results, we study the symplectic compactifications of the intermediate jacobian fibrations associated to a smooth cubic fourfold, as constructed by \cite{Sacca:boh}. In particular, we focus on the uniqueness of such compactifications. In the following $V$ is a smooth cubic fourfold and $U_1\subset\PP\oH^0(V,\cO_V(1))^*$ is the open subset containing linear sections with at worst one node. We denote by $\cJ_{U_1}(V)$ the associated intermediate jacobian fibration (see, for example, \cite{LSV}).
\begin{thm*}[Theorem~\ref{thm:una sola LSV}]
Let $V$ be a smooth cubic fourfold outside the Hassett divisors $\mathcal{C}_8$ and $\mathcal{C}_{12}$. Then there is a unique irreducible holomorphic symplectic compactification of $\mathcal{J}_{U_1}(V)$.
\end{thm*}

\subsection*{Structure of the paper}
In Section~\ref{sec:preliminaries} we recall the main definitions and main results that we will need later in the paper. 
Section~\ref{sec:lagr fibr} deals with lagrangian fibrations and here we prove Theorem~\ref{thm:lagr} and Corollary~\ref{cor:weak splitting}.
In Section~\ref{sec:movable} we classify prime exceptional divisors and prove Theorem~\ref{thm:movable} about the movable cone. 
Section~\ref{sec:examples} containes example of wall divisors and of birational morphisms between moduli spaces of sheaves, where these divisors arise.
Section~\ref{sec:ample} completes the classification of wall divisors and proves Theorem~\ref{thm:ample}. 
In Section~\ref{sec:counter-counter} we give a geometric and conceptual explanation of what was wrong in \cite[Theorem~3.3]{monwrong}: in that paper it was erroneously proved that the monodromy group of manifolds of OG10 type is not maximal.
Finally, Section~\ref{sec:IJac} contains applications of the previous results to the case of compactified intermediate jacobian fibration of smooth cubic fourfolds.

\subsection*{Acknowledgements}
We would like to thank Valeria Bertini, Antonio Rapagnetta, Ulrike Rie\ss$\,$ and Giulia Saccà for useful conversations. Both authors were partially supported by ``Progetto di ricerca INdAM per giovani ricercatori: Pursuit of IHS''. GM is member of the INDAM-GNSAGA and received support from it. CO thanks the Research Council of Norway project no. 250104 for financial support.


\section{Preliminaries}\label{sec:preliminaries}
Let $S$ be a projective K3 surface and $v\in\oH^{\operatorname{even}}(S,\ZZ)$ an effective and positive Mukai vector (see \cite[Definition~0.1]{Yoshioka:ModuliSpaces}). If $H\in\Pic(S)$ is an ample class, we denote by $M_v(S,H)$ the moduli space of Gieseker $H$-semistable sheaves $F$ on $S$ such that $ch(F)\sqrt{td_S}=v$.

\begin{thm}[\protect{\cite{LehnSorger}, \cite{rapagnetta:og10}, \cite{PeregoRapagnetta:OLS}, \cite{PeregoRapagnetta:BBFlattice}}]\label{thm:og10prelim}
Suppose that $H$ is $v$-generic (see \cite[Section~2.1]{PeregoRapagnetta:OLS}) and that $v=2w$ with $w^2=2$. Then the following hold.
\begin{enumerate}
\item The moduli space $M_v(S,H)$ has either locally factorial or $2$-factorial singularities and there exists a symplectic desingularisation 
$$\pi_v\colon\widetilde{M}_v(S,H)\longrightarrow M_v(S,H).$$
Moreover, if $\Sigma_v$ denotes the singular locus of $M_v(S,H)$, then $\widetilde{M}_v(S,H)$ is obtained by blowing up $\Sigma_v$ with its reduced scheme structure.\\
\item $\widetilde{M}_v(S,H)$ is an irreducible holomorphic symplectic manifold.\\
\item There is an isometry
$$\oH^2(\widetilde{M}_v(S,H),\ZZ)\cong U^{3}\oplus E_8(-1)^{2}\oplus A_2(-1),$$
where $U$ is the hyperbolic plane and $E_8$ and $A_2$ are the lattices associated to the corresponding Dynkin diagrams. Moreover,
$$v^\perp\cong\oH^2(M_v(S,H),\ZZ)\longrightarrow\oH^2(\widetilde{M}_v(S,H),\ZZ)$$
is injective, the second map being the pullback $\pi_v^*$. Finally, if $\alpha\in v^\perp$ has divisibility $2$, we have 
$$\frac{\alpha\pm\widetilde{\Sigma}}{2}\in \oH^2(\widetilde{M}_v(S,H),\ZZ)$$
\end{enumerate}
\end{thm}
Any irreducible holomorphic symplectic manifold deformation equivalent to a smooth moduli space as in the Theorem is called of OG10 type. We recall in the following example the notation for the moduli spaces used by O'Grady in his original construction.
\begin{ex}[\cite{O'Grady:OG10}]\label{ex:M_S}
Assume $v=(2,0,-2)$. In this case an ample class $H$ is generic if the only divisor $D$ such that $D^2\geq-4$ and $(D,H)=0$ is the trivial divisor (cf.\ \cite{O'Grady:OG10}). Fixed such a generic ample class, we denote by $M_S$ the moduli space $M_v(S,H)$. The locus $B$ of non-locally free sheaves is a Weil divisor that is not Cartier (cf.\ \cite{Perego}). If we denote by $\widetilde{B}$ its strict transform and by $\widetilde{\Sigma}$ the exceptional divisor of the desingularisation, then
$$ \oH^2(\widetilde{M}_S,\ZZ)\cong\oH^2(S,\ZZ)\oplus\langle\widetilde{B},\widetilde{\Sigma}\rangle,$$
where $\widetilde{B}^2=-2$, $\widetilde{\Sigma}^2=-6$ and $(\widetilde{B},\widetilde{\Sigma})=3$ (cf.\ \cite{rapagnetta:og10}).
Finally, $$v^\perp=\oH^2(S,\ZZ)\oplus\ZZ(1,0,1),$$ where the class $(1,0,1)$ corresponds to the Cartier divisor $2B$, and $\pi_v^*(2B)=2\widetilde{B}+\widetilde{\Sigma}$.
\end{ex}
\begin{ex}[Moduli spaces of torsion sheaves]\label{ex:torsion}
Let $S$ be a very general K3 surface of genus 2, that is $\Pic(S)=\ZZ H$ where $H$ is an ample line bundle of degree $2$. Fix the Mukai vector $v=(0,2H,s)$. Then the moduli space $M_v=M_v(S,H)$ parametrises sheaves of pure dimension 1. More precisely, $M_v$ containes the relative Picard variety $\Pic^{s+4}(\mathcal{C}/|2H|)$, where $\mathcal{C}\to U\subset|2H|$ is the universal family of smooth curves contained in the linear system $|2H|$.

There is a natural morphism $p\colon M_v\to|2H|$ that assigns to a sheaf its Fitting support; we call this map the \emph{Fitting morphism}.

Notice that $M_v$ is smooth when $s$ is odd, and it is singular when $s$ is even. The last case is the one we are interested in this paper, so from now on we suppose that $s$ is even. If $\pi\colon\widetilde{M}_v\to M_v$ is the symplectic desingularisation described in Theorem \ref{thm:og10prelim}, then the composition $p\circ\pi\colon\widetilde{M}_v\to|2H|$ is a lagrangian fibration.
\end{ex}
\begin{oss}\label{rmk:intersezione sopra e sotto}
Let $\pi_v\colon\widetilde{M}_v(S,H)\longrightarrow M_v(S,H)$ be the symplectic desinuglarisation of theorem \ref{thm:og10prelim}. We remark that the inclusion $\pi_v^*\colon\oH^2(M_v(S,H),\ZZ)\to\oH^2(\widetilde{M}_v(S,H),\ZZ)$ preserves the intersection product in the following sense. Let $C\subset M_v(S,H)$ be a curve, $\widetilde{C}\subset\widetilde{M}_v(S,H)$ its strict transform  and $D\in\Pic(M_v(S,H))$ a divisor. Then by the projection formula
$$ \widetilde{C}.\pi_v^*D=C.D.$$
\end{oss}
We recall that a monodromy operator is an isometry of the lattice $\oH^2(X,\ZZ)$ arising as parallel transport inside a deformation family of $X$. Such isometries are always contained in the index two subgroup $\Or^+(\oH^2(X,\ZZ))$ of orientation preserving isometries.
The main tool we are going to use throughout the rest of the paper is the following description of the monodromy group.
\begin{thm}[\cite{ono}]\label{monodromy}
Let $X$ be an irreducible holomorphic symplectic manifold of OG10 type. Then
$$\Mon^2(X)=\Or^+(\oH^2(X,\ZZ)).$$
\end{thm}

As in the case of K3 surfaces, interesting properties of divisors are preserved by monodromy transformations. There are two such classes of divisors, the first are wall divisors, which can roughly be seen as divisors whose orthogonal separates ample chambers of different birational models (see \cite{Mongardi:Cones}) or as divisors dual to extremal rays (see \cite{amver})\footnote{Both statements are true up to the action of monodromy Hodge isometries.}. The second one are stably prime exceptional divisors, which can be seen as divisors which are generally contractible in their Hodge locus or as divisors who can be deformed to an effective integral and negative divisor.
\begin{defn}
Let $X$ be an irreducible holomorphic symplectic manifold and let $D\in \Pic(X)$ be a divisor. The divisor $D$ is prime exceptional if it is effective, integral and $q(D)<0$. A divisor is stably prime exceptional if it is prime exceptional on a very general deformation of the pair $(X,D)$.
\end{defn}
\begin{defn}
Let $X$ be an irreducible holomorphic symplectic manifold and let $D\in \Pic(X)$ be a primitive divisor. The divisor $D$ is a wall divisor if $q(D)<0$ and, for all $f\in Mon^2(X)\cap Hdg^2(X)$, we have 
$$ f(D)^\perp \cap \mathcal{BK}_X=\emptyset. $$
\end{defn}
In particular, a stably prime exceptional divisor is the multiple of a wall divisor and wall divisors which have an effective multiple are stably prime exceptional divisors. By taking together the results of \cite{amver,bht,mark_tor,Mongardi:Cones} we have the following:
\begin{thm}
Let $X$ be a irreducible holomorphic symplectic manifold and let $D\in \Pic(X)$ be a divisor. Let $Y$ be an irreducible holomorphic symplectic manifold deformation equivalent to $X$ and let $D'\in \Pic(Y)$ be the parallel transport of $D$. Then, $D'$ is a wall divisor (respectively a stably prime exceptional divisor) if and only if $D$ is. 
\end{thm}

An intensive study of the birational geometry of (singular) moduli spaces of O'Grady type has been carried out by Meachan and Zhang \cite{Meach_Zhang}, they worked in the more general context of Bridgeland stability conditions and moduli spaces of stable objects in the derived category of a K3 surface, we will stick with the notion of Gieseker stability and moduli spaces of sheaves, and our (Bridgeland) stability condition will be a (Gieseker) stability condition parametrized by the choice of an ample class. For every $v$-generic stability condition $H$, there is an open set $\mathcal{C}$ of the space of stability conditions where all moduli spaces are isomorphic. This is called a Chamber of the space of stability conditions (in both senses).  

\begin{thm}\cite[Proposition 5.2 and Theorem 5.3]{Meach_Zhang}\label{thm:mz1}
Let $S$ be a K3 surface and let $v=2w$ be a Mukai vector of square $-8$. Let $H$ be a $v$-generic ample line bundle on $S$ and let $M_v(S,H)$ be the moduli space of semistable sheaves with Mukai vector $v$ and stability condition $H$. Let $\mathcal{C}$ be the chamber containing $H$ and let $H_0$ be an ample line bundle on $S$ in $\overline{\mathcal{C}}\setminus\mathcal{C}$.
Then the following hold:
\begin{enumerate}
\item There is a contraction map $\pi:\,M_v(S,H) \rightarrow M_v(S,H_0)$ which contracts precisely the locus of $\mathcal{S}$ equivalent $H_0$ semistable objects of $M_v(S,H)$.
\item The map $\pi$ is a divisorial contraction if and only if there exists a class $s\in H^{2*}(S,\ZZ)$ of type 1,1 with $s^2=-2$ and $(s,v)=(s,H_0)=0.$
\item The map $\pi$ is a small contraction if and only if there exists a class $s\in H^{2*}(S,\ZZ)$ of type 1,1 with $s^2=-2$, $(s,H_0)=0$ and $0<(s,v)\leq4.$
\end{enumerate} 
\end{thm}

Let $Stab^\dagger(S)$ denote the space of (Bridgeland) stability conditions on $D^b(S)$ and let $\sigma\in Stab^\dagger(S)$ be general.
\begin{thm}\cite[Theorem 7.6]{Meach_Zhang}\label{thm:mz2}
There is a globally defined and continuous map $l\colon Stab^\dagger(S)\rightarrow NS(M_v(S,\sigma))$. The map is independent of $\sigma$ and the image of a generic stability condition $\tau$ is the birational model of $M_v(S,\sigma)$ given by $M_v(S,\tau)$. Moreover, $l$ is surjective on big and movable divisors on $M_v(S,\sigma)$ and every $\QQ$ factorial $K$ trivial model of $M_v(S,\sigma)$ which is isomorphic to it in codimension one arises as a moduli space $M_v(S,\tau)$ for some generic $\tau\in Stab^\dagger(S)$.
\end{thm}

We will need two well known results in lattice theory, which both go under the name of Eichler's criterion. For any even lattice $L$, we can define the discriminant group $A_L:=L^\vee/L$ and this inherits a quadratic form from $L$ with values in $\mathbb{Q}/2\mathbb{Z}$. Any isometry of $L$ has then an induced action on $A_L$, and the kernel of this map is the subgroup $\widetilde{O}(L)$.
Moreover, to any element $v\in L$ we can define its divisibility $\div(v)$ as the positive generator of the ideal $(v,L)$. This gives a natural map $L\rightarrow A_L$ sending $v$ to $[v/\div(v)]$. We will make use of two versions of what is classically known as "Eichler's Lemma'', see \cite[Lemma 3.5]{GHS10}.
\begin{lem}\label{lem:eichler}
Let $L'$ be an even lattice and let $L=U^2\oplus L'$. Let $v,w\in L$. Suppose in addition that
\begin{itemize}
\item $v^2=w^2$,
\item $[v/\div(v)]=[w/\div(w)]$ in $A_L$.
\end{itemize}
Then there is an isometry in $\widetilde{SO}^+(L)$ sending $w$ to $v$.
\end{lem}
\begin{lem}\label{lem:eichler_gen}
Let $L'$ be an even lattice and let $L=U^n\oplus L'$. Let $S,T$ be two sublattices of $L$ such that the following holds:
\begin{itemize}
\item There are two basis of $S=\langle s_1,\dots, s_r \rangle$ and $T=\langle t_1,\dots, t_r\rangle$ of the same cardinality, and $r\leq n-1$. 
\item For all $i,j$ we have $(s_i,s_j)=(t_i,t_j)$.
\item For all $i$, we have $[s_i/\div(s_i)]=[t_i/\div(t_i)]$.
\end{itemize}
Then there is an isometry in $\widetilde{SO}^+(L)$ sending $S$ to $T$.
\end{lem}
Clearly, the second Lemma is a generalization of the first.


\section{Lagrangian fibrations}\label{sec:lagr fibr}

\begin{lem}\label{lem:lagrorbit}
Let $l\in L:=U^3\oplus E_8(-1)^2 \oplus A_2(-1)$ be a primitive  element of square zero. Then $div(l)=1$ and there is a single orbit for the action of $O^+(L)$.
\end{lem}
\begin{proof}
The discriminant group of $L$ is of three torsion, hence $l$ can have only divisibility one or three.
Inside $A_2(-1)$, any element $s$ of divisibility three has the form $e_1+2e_2+3t$ or $-e_1-2e_2+3t$, where $e_1$ and $e_2$ are the standard generators with $e_1^2=e_2^2=-2$ and $\langle e_1, e_2 \rangle=1$ and $t$ is any element. It follows that, modulo $18$, the square of such an element is $-6$. 
Therefore, a primitive element of divisibility $3$ inside $L$ is an element of the form $3w+as$ for some $w\in U^3\oplus E_8(-1)^2$, $a$ not divisible by three and $s\in A_2(-1)$ primitive of divisibility three. The square of such an element is congruent to $-6a^2$ modulo $18$, which cannot be zero. Therefore $l$ has divisibility one and, by Lemma \ref{lem:eichler}, the action of $O^+(L)$ has a single orbit.
\end{proof}
The above Lemma is the technical core of the section, as for any irreducible holomorphic symplectic manifold $X$, if $p\colon X\to \mathbb{P}^n$ is a lagrangian fibration, the divisor $p^*(\mathcal{O}(1))$ is primitive, nef and isotropic. In particular, if a divisor is induced by a lagrangian fibration on a different birational model of $X$, it will be isotropic and in the boundary of the Birational K\"ahler cone. The following is a converse for manifolds of OG10 type:

\begin{thm}\label{thm:lagr}
Let $X$ be an irreducible holomorphic symplectic manifold of OG10 type and let $\mathcal{O}(D)\in Pic(X)$ be a non-trivial line bundle whose Beauville-Bogomolov  square is $0$.
Assume that the class $[D]$ of $\mathcal{O}(D)$ belongs to the boundary of the birational K\"ahler cone of $X$. 

Then, there exists  a smooth irreducible holomorphic symplectic manifold $Y$ and a bimeromorphic map 
$\psi\colon Y\dashrightarrow X$ such that $\mathcal{O}(D)$ induces  a lagrangian fibration $p\colon Y\rightarrow \mathbb{P}^5$. 
Moreover, smooth fibres of $p$ are, up to isogeny, principally polarized abelian fivefolds.
\end{thm}
\begin{proof}
The proof of this theorem is completely analogous to \cite[Theorem 7.2]{MR} for the case of O'Grady type sixfolds, for the convenience of the reader we sketch it here.
By the work of Matsushita \cite[Theorem 1.2]{matsu_lagr} and Wieneck \cite[Theorem 1.1]{wie}, the statement of the Theorem either holds for all deformations of the pair $(X,\mathcal{O}(D))$ where the parallel transport of $[D]$ belongs to the boundary of the birational K\"ahler cone, or never holds.
By Lemma \ref{lem:lagrorbit} and \cite[Section 5.3 and Lemma 5.17(ii)]{mar_prime}, the moduli space of pairs $(X,\mathcal{O}(D))$ with $D$ primitive and isotropic is connected. 
Therefore, we only need to provide an example of a lagrangian fibration with principally polarized fibres, but such an example is well known and represented by moduli spaces $M_v(S,H)$ of torsion sheaves as in Example~\ref{ex:torsion}. Notice that the general fibre of the Fitting morphism is isomorphic to a jacobian of genus five curves, therefore it is (up to isogeny) a principally polarized abelian variety of dimension 5.
\end{proof}
The above Theorem has an immediate consequence concerning the movable cone of a projective manifold of OG10 type:
\begin{cor}
Let $X$ be a projective manifold of OG10 type. Then the movable cone of $X$ is $\overline{\mathcal{BK}}_X\cap H^{1,1}(X,\mathbb{Z})$.
\end{cor}
\begin{proof}
By \cite[Theorem 7]{hass_ts_moving}, the closure of the movable cone equals $\overline{\mathcal{BK}}_X\cap H^{1,1}(X,\mathbb{Z})$.
By \cite[Corollary 15]{hass_ts_moving}, all elements of $\overline{\mathcal{BK}}_X\cap H^{1,1}(X,\mathbb{Z})$ with positive square are in the movable cone. By Theorem~\ref{thm:lagr}, also isotropic elements in this intersection are in the movable cone.
\end{proof}
Notice in particular that the movable cone is closed.\\
By using a result of Riess \cite[Theorem 4.2]{rie}, we obtain as a corollary that the Weak Splitting property conjectured by Beauville \cite{beau07} holds when the manifold has a square zero divisor.
\begin{cor}\label{cor:weak splitting}
Let $X$ be a projective irreducible holomorphic symplectic manifold of OG10 type and let $D$ be an isotropic divisor on it. Let $DCH(X)\subset CH_{\mathbb{Q}}(X)$ be the subalgebra generated by divisor classes. Then the restriction of the cycle class map $$cl_{|DCH(X)}\colon DCH(X)\to \oH^*(X,\mathbb{Q})$$ is injective.
\end{cor}
\begin{proof}
\cite[Theorem 4.2]{rie} proves that the Weak Splitting property holds for all manifolds $X$ such that one of their birational model has a lagrangian fibration. By \cite[Section 6]{mark_tor}, a manifold with a square zero divisor has a square zero divisor in the boundary of the birational K\"ahler cone and from Theorem \ref{thm:lagr} it follows that $X$ has a birational model with a lagrangian fibration.   
\end{proof} 

\section{The birational K\"ahler cone}\label{sec:movable}

\begin{prop}\label{prop:pex}
Let $X$ be a manifold of OG10 type and let $D\in Pic(X)$ be a primitive divisor. Then a multiple of $D$ is stably prime exceptional if and only if $D^2=-2$ or $D^2=-6$ and $div(D)=3$.
\end{prop}
\begin{proof}
If $E$ is a stably prime exceptional divisor and $D$ is a primitive divisor such that $E=aD$, Markman \cite{mar_prime} proves that the reflection $R_E=R_D$ which sends a divisor $F$ to $F-2\frac{(D,F)}{D^2}D$ is a monodromy operator, and in particular it is integral. Therefore, $2div(D)/D^2$ is an integer. As $H^2(X,\mathbb{Z})\cong U^3\oplus E_8(-1)^2\oplus A_2(-1)$, $div(D)$ can be either $1$ or $3$. This in turn implies $D^2=-2$ in the first case and $D^2=-6$ in the second case. Being a stably prime exceptional divisor is a property which is invariant under the monodromy group by \cite[Section 6]{mark_tor}, and elements of square $-2$ form a single monodromy orbit by Lemma \ref{lem:eichler}, as do elements of square $-6$ and divisibility $3$. Therefore, to prove our claim we only need an example of a manifold $X$ with two prime exceptional divisors having degree $-2$ and degree $-6$, respectively.

Let $S$ be a projective K3 surface and $M_S$ the moduli space of semistable rank $2$ sheaves with trivial first Chern class and second Chern class of degree $4$. The locus $B$ parametrising non-locally free sheaves is a (Weil) divisor. Take $X=\widetilde{M}_S$ to be O'Grady's symplectic desingularisation; then the strict transform $\widetilde{B}$ of $B$ and the exceptional divisor $\widetilde{\Sigma}$ of the desingularisation are prime exceptional divisors with the right degrees and divisibilities, as written in Example \ref{ex:M_S}, by \cite{rapagnetta:og10}.

\end{proof}
\begin{thm}\label{thm:movable}
Let $X$ be a manifold of OG10 type. Then, the birational K\"ahler cone of $X$ is an open set inside one of the connected components of $$\mathcal{C}(X)\setminus \bigcup_{D^2=-2 \text{ or } (D^2=-6 \text{ and } div(D)=3)} D^\perp.$$
\end{thm}
\begin{proof}
By \cite[Section 5]{mark_tor}, the b
irational K\"ahler cone is an open set in the fundamental exceptional chamber, which is a connected component of $$\mathcal{C}(X)\setminus \bigcup_{D \text{ stably prime exceptional }} D^\perp.$$
By Proposition \ref{prop:pex}, stably prime exceptional divisors are multiples of divisors of square $-2$ or of square $-6$ and divisibility $3$, therefore the claim follows. 
\end{proof}
In particular, the above Theorem readily implies that the two families of O'Grady tenfolds constructed by Laza, Saccà and Voisin in \cite{LSV} and in \cite{VoisinLSV} are not birational for a very general cubic fourfold $Y$, results that has been proven (with different methods) in \cite{Sacca:boh}. We refer to Section~\ref{sec:IJac} for the notation.
\begin{cor}
Let $V$ be a very general cubic fourfold (in the sense of Hassett) and let $\IJac(V)$ and $\IJac^t(V)$ be the (compactified and smooth) families of intermediate jacobians (resp.\ twisted intermediate jacobians). Then $\IJac(V)$ and $\IJac^t(V)$ are not birational.
\end{cor}
\begin{proof}
If $V$ is very general, then by \cite[Proposition~4.1]{ono} (see Lemma~\ref{prop:primitive_cohom}) $\Pic(\IJac(V))\cong U$, where $U$ is the unimodular hyperbolic plane. Both $\IJac(V)$ and $\IJac^t(V)$ have a natural structure of lagrangian fibrations over $\mathbb{P}^5$, therefore they have a nef square zero divisor in their Picard lattice by \cite{matsu_lagr}. For general $V$, Voisin proved that $\IJac(V)$ and $\IJac^t(V)$ are not birational as lagrangian fibrations, therefore if they were birational there would be two square zero movable divisors in $\Pic(\IJac(V))$. By fixing standard square zero generators of $\Pic(\IJac(V))=U=\langle e,f\rangle$, all square zero classes are $e,f,-e,-f$ and, up to a sign choice, positive ones are only $e$ and $f$. However, the class $e-f$ is negative on $e$ and positive on $f$, therefore by Theorem \ref{thm:movable} only one among $e$ and $f$ can be movable, which implies that the two manifolds are not birational.  
\end{proof}


\section{Examples of wall divisors}\label{sec:examples}
In this section we exhibit explicit examples of wall divisors. The following remark is fundamental for the computations. If $X$ is an irreducible holomorphic symplectic manifold and $\oH_2(X,\ZZ)$ is the homology group of curve classes, then Poincar\'e duality and the non-degeneratness of the Beauville--Bogomolov--Fujiki form give the embedding
$$\oH_2(X,\ZZ)\cong\oH^2(X,\ZZ)^*\hookrightarrow\oH^2(X,\QQ).$$
If $R\in\oH_2(X,\ZZ)$ is the class of an extremal ray of the Mori cone of curves, then by \cite[Lemma~1.4]{Mongardi:Cones} an integral generator $D$ of the line $\QQ R\subset\oH^2(X,\QQ)$ is the class of a wall divisor, and $D^\vee:=D/\div(D)=R$. The strategy is then to look for lines describing certain Mukai flops and write down their classes in the group $\oH^2(X,\QQ)$.

We refer to Example~\ref{ex:M_S} and Example~\ref{ex:torsion} for the notation and background.

\subsection{Example: the zero section}\label{example:zero section}
Let $S$ be a very general K3 surface of genus $2$, that is $\Pic(S)=\ZZ H$ with $H^2=2$. We work with the moduli space $M_{(0,2H,-4)}$ and its desingularisation $\widetilde{M}_{(0,2H,-4)}$. The generic point of $M_{(0,2H,-4)}$ is of the form $i_*L$, where $i\colon C\to S$ is the inclusion of a genus $5$ curve such that $C\in|2H|$ and $L$ is a degree $0$ line bundle on $C$. In particular, the generic fibre of the Fitting-support morphism $p\colon M_{(0,2H,-4)}\to|2H|$ is an abelian variety, so that there is a generic zero section $s\colon|2H|\dashrightarrow M_{(0,2H,-4)}$. We denote by $Z_S$ the closure in $\widetilde{M}_{(0,2H,-4)}$ of the image of $s$.

We want to take a general line $l\subset Z_S$ and compute its class in $\Pic(\widetilde{M}_{(0,2H,-4)})_{\QQ}=\langle a,b,\sigma\rangle$, where $a=(-1,H,0)$, $b=(0,0,1)$ are the generators of $v^\perp_{\operatorname{alg}}$, and $\sigma$ is the class of the exceptional divisor of the desingularisation. The line $l$ is an extremal ray of the Mori cone of $\widetilde{M}_{(0,2H,-4)}$, so that the divisors $D$ such that $l=D^\vee$ is a wall divisor.

We use the horizontal curve $l$ defined in \cite[Section~4.1.1]{ono}; the following result is shown there.
\begin{lem}[\protect{\cite[Remark~4.9, Lemma~4.10]{ono}}]
$l.a=-1$, $l.b=1$ and $l.\sigma=0$.
\end{lem}
As a consequnce we get that $l=a-3b$. In particular it is already integral, so that the associated wall divisor is $D=a-3b$. Notice that $D$ has degree $-4$ and divisibility $1$.

\subsection{Example: $\PP^5$}\label{example:lP^5}
Let $S$ be an elliptic K3 surface such that $\Pic(S)=\langle e,f\rangle$, where $e-f$ is the class of a section and $f$ is the class of a fibre; in particular $e^2=0=f^2$ and $(e,f)=1$. 
The class $H=e+3f$ is ample and generic (in the sense of Example~\ref{ex:M_S}); $\widetilde{M}_S(H)$ is the associated smooth O'Grady moduli space.

Consider the class $H_0=e+2f$; then $H_0$ is ample but not generic. In fact one can describe explicitly the singular locus of $M_S(H_0)$: a semistable sheaf $F$ is singular in $M_S(H_0)$ if either $F$ is strictly $H$-semistable or $F$ fits in a short exact sequence 
\begin{equation}\label{eqn:H_0}
0\longrightarrow L\longrightarrow F\longrightarrow L^\vee\longrightarrow0,
\end{equation}
where $L=e-2f$. This follows directly from the proof of \cite[Lemma~1.1.5]{O'Grady:OG10}. Notice that a sheaf $F$ fitting in a sequence like (\ref{eqn:H_0}) is locally free, and there is a $\PP^5$ of such extensions. In particular, blowing up the locus $\Sigma(H_0)$ of strictly $H$-semistable sheaves produces a symplectic variety $\widetilde{M}_S(H_0)$ (which is still singular).

As explained in \cite[Proposition~4.4 and Corollary~4.6]{MR}, any $H$-semistable sheaf remains $H_0$-semistable, so there is a regular contraction morphism $c\colon M_S(H)\to M_S(H_0)$. Moreover, this morphism lifts to a regular morphism between the corresponding blow-ups, i.e.\ there is a commutative diagram
\begin{equation*}
\xymatrix{
\widetilde{M}_S(H)\ar@{->}[r]^{\tilde{c}}\ar@{->}[d] & \widetilde{M}_S(H_0)\ar@{->}[d] \\
M_S(H)\ar@{->}[r]^{c}                                         &  M_S(H_0),
}
\end{equation*}
and $\tilde{c}$ is a flopping contraction. The extremal curve $R$ associated to the contraction $\tilde{c}$ is any line inside the $\PP^5$ of extensions of the form (\ref{eqn:H_0}). We claim that the divisor $D\in\Pic(\widetilde{M}_S(H))$ such that $D^\vee=R$ is the class $e-2f$, seen as a divisor class in $\Pic(\widetilde{M}_S(H))$ via the Mukai--Donaldson--Le Poitier morphism. In particular, $D=e-2f$ is a wall divisor of degree $-4$.

First of all, we notice that the contracted $\PP^5$ is contained in the locally free locus, so that the curve $R$ is disjoint from both the exceptional divisor $\widetilde{\Sigma}$ and the divisor $\widetilde{B}$ of non-locally free sheaves (cf.\ Example~\ref{ex:M_S}). In particular, we can suppose that $[R]\in\oH^2(S,\ZZ)\cong\langle\widetilde{B},\widetilde{\Sigma}\rangle^\perp$. On the other hand, such a class must be orthogonal to the ample class $H_0$ by construction. It follows that, up to a constant, $R=e-2f$, and we are done.


\subsection{Example: $\PP^3$-bundle}\label{example:lP^3-bundle}
In this section we work with the symplectic resolution $\widetilde{M}_{v}(S,H)$ of the moduli space $M_{v}(S,H)$, where $(S,H)$ is a polarised K3 surface of genus $2$ and $v=(0,2H,2)$. Moreover, we assume that $(S,H)$ is very general, that is $\Pic(S)=\ZZ H$.

We start by defining a closed subvariety $Y\subset M_{v}(S,H)$ of codimension $3$ such that $Y$ is isomorphic to a $\PP^3$-bundle over $\Hilb^2(S)$. Then we identify the class of the extremal ray corresponding to a line in a $\PP^3$-fibre of $Y$.

Let $w=(1,2H,3)$ and notice that $M_w(S,H)\cong\Hilb^2(S)$. The moduli space $M_w(S,H)$ parametrises sheaves of the form $I_{\xi}(2)$, where $\xi\in\Hilb^2(S)$ and $I_{\xi}$ is the corresponding sheaf of ideals. By a direct computation, we get that $h^0(I_\xi(2))=4$ and $h^1(I_\xi(2))=h^2(I_\xi(2))=0$. If $s\in\PP\oH^0(I_\xi(2))$ is a section, then the sheaf $F_s$ fitting in the short exact sequence
\begin{equation}\label{eqn:Y}
0\longrightarrow\cO_S\stackrel{s.}{\longrightarrow} I_\xi(2)\longrightarrow F_s\longrightarrow0
\end{equation}
belongs to the moduli space $M_v(S,H)$.
The subvariety $Y\subset M_v(S,H)$ formed by sheaves arising as in (\ref{eqn:Y}) is by construction a $\PP^3$-bundle over $M_w(S,H)$.

Let $\widetilde{M}_v(S,H)$ be the symplectic desingularisation of $M_v(S,H)$.
We want to determine the class in $\oH_2(\widetilde{M}_v(S,H),\ZZ)\subset\oH^2(\widetilde{M}_v(S,H),\QQ)$ of a line $R\cong\PP^1$ in the fibre of the $\PP^3$-bundle $Y$. Then a primitive integral generator $D$ of the line $\QQ R$ is a wall divisor.

From now on we fix a very general point $\xi\in\Hilb^2(S)$; the support of $\xi$ is composed by two reduced and disjoint points, i.e.\ $\operatorname{Supp}(\xi)=\{p,q\}$ and $p\neq q$.
If $s\in\PP\oH^0(I_\xi(2))$ is a section, we denote by $C_s$ the zero locus of $s$. With an abuse of notation, we identify $s$ with $C_s$ when no confusion arises.
\begin{cla}\label{claim:L}
There exists a pencil $L\subset\PP\oH^0(I_\xi(2))$ such that
$C_s$ is smooth for all but a finite number of sections $s_1,s_2,\cdots,s_k\in L$ and moreover:
\begin{itemize}
\item $C_{s_1}$, $C_{s_2}$ and $C_{s_3}$ have two smooth irreducible components; more precisely $C_{s_i}=C_{s_i,1}\cup C_{s_i,2}$ with $C_{s_i,j}\in|H|$ smooth.
\item $C_{s_i}$, $i=4,\cdots,k$, is irreducible.
\end{itemize}
\end{cla}
\begin{proof}
We can think $\PP\oH^0(I_\xi(2))$ as the subset of $|2H|$ of curves passing through the points $p$ and $q$, support of $\xi$. Recall that the K3 surface $S$ is the double cover of $\PP^2$ ramified along a smooth sextic curve $\Gamma$; the linear system $|2H|$ is then isomorphic to the linear system $|\mathcal{O}_{\PP^2}(2)|$ of conics in $\PP^2$. A curve $C\in|2H|$ is smooth only if the corresponding conic in $|\mathcal{O}_{\PP^2}(2)|$ is smooth and is not tangent to the sextic $\Gamma$. Moreover, the locus of singular conics in $|\mathcal{O}_{\PP^2}(2)|$ is a degree $3$ hypersurface. 

Let $d\colon S\to\PP^2$ be the cover. By the generality of $\xi$, the image $d(\xi)$ consists of two disjoint points and $\PP\oH^2(I_\xi(2))$ is isomorphic to $\PP\oH^0(I_{d(\xi)}(2))\subset|\mathcal{O}_{\PP^2}(2)|$. The locus of singular conics in $\PP\oH^0(I_{d(\xi)}(2))$ is then again a hypersurface of degree $3$. It follows that a general pencil $L'$ in $\PP\oH^0(I_{d(\xi)}(2))$ contains three singular members: these correspond to the three reducible curves in $L=f^*L'$. 
Now, let $\mathcal{C}\subset\Gamma\times\PP\oH^0(I_{d(\xi)}(2))$ be the incidence variety of conics tangent to $\Gamma$ at some point. The projection $\mathcal{C}\to\PP\oH^0(I_{d(\xi)}(2))$ is finite on its image. The projection $\mathcal{C}\to\Gamma$ has generic fibre of dimension $1$. So the pencil $L'$ contains only finitely many conics tangent to $\Gamma$ and the claim follows.
\end{proof}

\begin{oss}\label{rmk:S_1+S_2}
The locus of reducible conics in $\PP\oH^0(I_{d(\xi)}(2))$ is the union $S_1\cup S_2$ of two surfaces. The surface $S_1$ parametrises reducible conics where both the points in the support of $d(\xi)$ are on one irreducible component. Since there exists a unique line passing through two points, it follows that $S_1$ is a linear surface. On the other hand, the surface $S_2$ parametrises reducible conics such that each irreducible component contains one point of the support of $d(\xi)$; $S_2$ has degree $2$.
\end{oss}

In order to define the line $R\subset M_v(S,H)$, we need to define a family $\mathcal{F}$ of sheaves on $L\times S$, flat over $L$, such that $\mathcal{F}_s\in M_v(S,H)$ for every $s\in L$.

Let $\pi_L$ and $\pi_S$ the projections from $L\times S$ to $L$ and $S$, respectively. 
The following result is an adaptation to our case of \cite[Section~2.2 and Appendix]{Perego}.
\begin{lem}
Let $\operatorname{P}=\PP\oH^0(I_\xi(2))$.
There exists an injective morphism 
$$\varphi\colon\pi_S^*\mathcal{O}_S\otimes\pi_{\operatorname{P}}^*\mathcal{O}_{\operatorname{P}}(-1)\longrightarrow\pi_S^*I_\xi(2)$$
defining a sheaf $\mathcal{F}':=\operatorname{coker}(\varphi)$ on $\operatorname{P}\times S$, flat over $\operatorname{P}$, such that $\mathcal{F}'_s$ is the sheaf in $M_v(S,H)$ corresponding to the section $s$ (cf.\ (\ref{eqn:Y})), for every $s\in \operatorname{P}$.
\end{lem}
\begin{proof} 
Viewing $\PP\oH^0(I_\xi(2))$ as a $\PP^3$-bundle over a point $\star=\operatorname{Spec}(\CC)$, we can write 
$$\operatorname{P}=\PP\oH^0(I_\xi(2))=\PP\left(o_*\mathcal{H}om(\mathcal{O}_S,I_\xi(2))\right),$$
where $o\colon S\to\star$ is the structure morphism of $S$. Denote by $p\colon \operatorname{P}\to\star$ the induced morphism. The universal property of $\PP$-bundles yields a canonical injective morphism $f\colon\mathcal{O}_{\operatorname{P}}(-1)\to p^*o_*\mathcal{H}om(\mathcal{O}_S,I_\xi(2))$. By the commutativity of the square 
\begin{equation*}
\xymatrix{
\operatorname{P}\times S\ar@{->}[r]^{q_{\operatorname{P}}}\ar@{->}[d]^{q_S} & Y\ar@{->}[d]^{p} \\
S\ar@{->}[r]^{o} & \star
}
\end{equation*}
we eventually get the isomorphism
$$ p^*o_*\mathcal{H}om(\mathcal{O}_S,I_\xi(2))\cong q_{\operatorname{P}*}\mathcal{H}om(q_S^*\mathcal{O}_S,q_S^*I_\xi(2)).$$
It follows that $f$ defines a section $\phi\in\oH^0(\operatorname{P}\times S, \mathcal{H}om(q_S^*\mathcal{O}_Sq_{\operatorname{P}}^*\mathcal{O}_{\operatorname{P}}(-1),q_S^*I_\xi(2)))$. 

By construction $\phi$ is defined fibrewise, and its restriction to each fibre is injective. It follows that $\phi$ is injective, so that $\mathcal{F}=\operatorname{coker}(\varphi)$ is flat over $\operatorname{P}$, and that $\mathcal{F}_s$ is the sheaf associated to the section $s\in\operatorname{P}$.
\end{proof}

Let $j\colon L\times S\to\operatorname{P}\times S$ be the inclusion and $\mathcal{F}:=j^*\mathcal{F}'$ the restriction. Then $\mathcal{F}$ is a sheaf on $L\times S$, flat over $L$.
The pair $(L,\mathcal{F})$ defines a line $R\subset M_v(S,H)$. Now, we want to understand the intersection of $R$ with $\Sigma$, the singular locus of $M_v(S,H)$.
\begin{lem}\label{lem:R.Sigma}
$R$ intersects $\Sigma$ in one point, corresponding to the reducible curve where $\xi$ is contained in only one irreducible component.
\end{lem}
\begin{proof}
Let $s\in \PP\oH^0(I_\xi(2))$ be a section, $C_s$ the associated curve and $F_s$ the associated sheaf. If $C_s$ is irreducible, then $F_s$ is stable. Assume then that $C_s=C_1\cup C_2$ is reducible. The stability of $F_s$ is checked by studying the sheaves $G_i:=(F_s|_{C_i})/\operatorname{tors}$. Since $F_s$ is defined by the short exact sequence (\ref{eqn:Y}), the sheaf $G_i$ is defined by
$$0\longrightarrow\cO_{S}\longrightarrow I_{\xi\cap C_i}(1)\longrightarrow G_i\longrightarrow0.$$
Using the same notation as in Remark~\ref{rmk:S_1+S_2}, if $s\in S_1$, then without loss of generality we can suppose that $\xi\subset C_2$. This implies that $I_{\xi\cap C_1}(1)=\cO_S(1)$ and we have a square 
\begin{equation*}
\xymatrix{
0\ar@{->}[r] & \mathcal{O}_S\ar@{->}[r]^{s.}\ar@{->}[d]^{id} & \mathcal{O}_S(1)\ar@{->}[r]\ar@{->}[d] & i_{1*}\mathcal{O}_{C_1}(2)\ar@{.>}[d]\ar@{->}[r] & 0 \\
0\ar@{->}[r] & \mathcal{O}_S\ar@{->}[r]^{s.} & I_\xi(2)\ar@{->}[r] & F_s\ar@{->}[r] & 0
}
\end{equation*}
with injective vertical arrows. Notice that the dotted arrow is induced by the commutativity of the first square, which is in turn induced by the inclusion $\oH^0(\mathcal{O}_S(1))\subset\oH^0(I_\xi(2))$ given by multiplication by the equation of $C_2$ (which is uniquely determined by $\xi$).

On the other hand, if $s\in S_2\setminus(S_1\cap S_2)$, then $I_{\xi\cap C_i}(1)=I_{p_i}(1)$ and one can directly check that neither of $G_i$ can destabilise $F_s$.

Finally, since $L$ is generic, the three reducible points $s_1$, $s_2$ and $s_3$ of Claim~\ref{claim:L} decompose as $s_1\in S_1$ and $s_2,s_3\in S_2$, from which the claim follows.
\end{proof}
Let now $\widetilde{R}$ be the strict transform of $R$ in $\widetilde{M}_v$, and $\pi\colon\widetilde{M}_v\to M_v$ the desingularisation morphism. We want to determine the coefficients of $\widetilde{R}$ in 
$$\Pic(\widetilde{M}_v)_{\QQ}=\operatorname{span}_{\QQ}\{a,b,\sigma\},$$ 
where $a=(2,H,0)$ and $b=(0,0,1)$ are the generators of $v^\perp_{\operatorname{alg}}$, and $\sigma$ is the class of the exceptional divisors of $\pi$.
\begin{lem}\label{claim}
$\widetilde{R}.a=2$ and $\widetilde{R}.b=1$.
\end{lem}
\begin{proof}
By Remark~\ref{rmk:intersezione sopra e sotto}, $\widetilde{R}.a=R.a$ and $\widetilde{R}.b=R.b$. 
The claim follows then from \cite[Theorem~8.1.5]{HL:ModuliSpaces} and the Grothendieck--Riemann--Roch formula, as explained in \cite[Lemma~4.10]{ono}.
\end{proof}
Together with Lemma~\ref{lem:R.Sigma}, this shows that
$$ \widetilde{R}=-\frac{1}{2}a-\frac{3}{2}b-\frac{1}{6}\sigma=-\frac{1}{2}(a+3b+\sigma)+\frac{1}{3}\sigma.$$
If we put $x=-\frac{1}{2}(a+3b+\sigma)$, then we notice that $x$ is integral and that $x^2=-4$. It follows that the divisor 
$$D:=3x+\sigma$$
is a wall divisor such that $D^\vee=\widetilde{R}$. Notice that $D^2=-24$ and $\div(D)=3$.

\begin{oss}
The projection of $D$ inside $\Pic(M_v)$ is $-a-3b$, which has square $-10$ and divisibility $2$, giving an example of \cite[Theorem~5.3, item (SC)]{Meach_Zhang}.
\end{oss}

\begin{oss}\label{rmk:lP^5 deformato}
There is a brational isomorphism 
$$ \widetilde{M}_{(0,2H,2)}\cong\widetilde{M}_{(0,2H,-2)}\dashrightarrow\widetilde{M}_{(0,2H,-4)}$$
induced by the unique $\mathfrak{g}^1_2$ on the smooth curves in $|2H|$ (see Example~\ref{example:hyperelliptic}). The first isomorphism is the one given by taking the tensor with $-H$. This birational morphism induces a morphism on the corresponding Picard groups
$$ \varphi\colon\Pic(\widetilde{M}_{(0,2H,2)})\longrightarrow\Pic(\widetilde{M}_{(0,2H,-4)})$$
that has been written down in coordinates in \cite[Section~4.1.2, Section~4.1.3]{ono}. It is easy then to see that $\varphi(x)$ coincides with the class of the wall divisor of Example~\ref{example:zero section}. Notice that $\varphi$ is a parallel transport operator (since it is induced by a birational isomorphism), therefore this implies that the class $x$ corresponds to the class of a wall divisor that is deformation of the zero section in $\widetilde{M}_{(0,2H,-4)}$.
\end{oss}


\subsection{Example: the O'Grady birational morphism}
Let $S$ be a projective K3 surface with a polarisation $H$ of degree $2$. We consider the moduli space $M_{(0,2H,2)}$, whose generic member is of the form $i_*L$, where $i\colon C\to S$ is the closed embedding of a smooth curve $C\in|2H|$ and $L$ is a degree $6$ line bundle on $C$.

O'Grady in \cite[Proposition~4.1.2]{O'Grady:OG10} defines a birational morphism 
$$\phi\colon M_{(0,2H,2)}\dashrightarrow M_{(2,2H,0)}$$
defined in the following way.
First of all, notice that if $C$ is smooth and $L$ is a line bundle of degree $6$ on $C$, then $h^0(i_*L)\geq2$. O'Grady defines then the open subset $\mathcal{J}^0\subset M_{(0,2H,2)}$ consisting of sheaves of the form $i_*L$, where $C$ is smooth and $L$ is a globally generated line bundle of degree $6$ such that $h^0(L)=2$. He defines $\phi(i_*L)$ as the dual of the kernel of the surjection $\oH^0(L)\otimes\cO_S\to i_*L$. As remarked in the proof of \cite[Lemma~4.20]{ono}, this birational morphism coincides with the birational morphism induced by the Fourier--Mukai transform with kernel the ideal sheaf of the diagonal in the product $S\times S$.

We denote by $\widetilde{\phi}\colon\widetilde{M}_{(0,2H,2)}\dashrightarrow\widetilde{M}_{(2,2H,0)}$ the birational morphism induced on the symplectic desingularisations.

The indeterminacy locus of $\phi$ (and of $\widetilde{\phi}$) is generically identified with the relative Brill--Noether locus $W=\mathcal{W}^6_2(|2H|)$ of line bundles $L$ of degree $6$ such that $h^0(L)=3$. Notice that $W$ has dimension $7$.
On the other hand, one can easily check that the generic member of the $\PP^3$-bundle $Y\subset M_{(0,2H,2)}$, defined before in Example~\ref{example:lP^3-bundle}, is of the form $i_*L$ with $h^0(L)=3$. It follows that $W=Y$ and so that $\widetilde{\phi}$ is the generalised Mukai flop around a $\PP^3$-bundle.


\subsection{Example: $\PP^3$-bundles over the locus of non-reduced curves}\label{subsec:example -24}
In this example we work with the symplectic resolution $\widetilde{M}_{(0,2H,-4)}(S,H)$ of the moduli space $M_{(0,2H,-4)}(S,H)$, where $(S,H)$ is a polarised K3 surface of genus $2$; moreover, we assume that $(S,H)$ is very general, that is $\Pic(S)=\ZZ H$. The moduli space $M_{(0,2H,-4)}$ comes with a lagrangian fibration structure $p\colon M_{(0,2H,-4)}\to |2H|$ that associates to each sheaf its Fitting support. Let $\Delta\subset|2H|$ be the locus of non-reduced curves, that is any curve in $\Delta$ is of the form $C=2C'$, where $C'\in|H|$.

Now, let us consider the restriction 
$$ p\colon M_\Delta\longrightarrow\Delta,$$
where $M_\Delta\subset M_{(0,2H,-4)}$ is the sublocus of sheaves whose Fitting support is not reduced. It is known (see for example \cite[Section~3.7]{DeCataldoRapagnetta:HodgeNumbersOG10}) that $M_\Delta$ has two irreducible components, denoted $M_1$ and $M_2$. The component $M_1$ parametrises sheaves whose schematic support is the reduced curve: these sheaves are of the form $i_*G$, where $i\colon C'\to S$ and $G$ is a rank $2$ vector bundle of degree $-2$. The component $M_2$ contains an open subset $M_2^0:=M_2\setminus(M_2\cap M_1)$ parametrising sheaves whose schematic support is the non-reduced curve itself. We recall that $M_2^0$ consists of stable sheaves, that is it does not intersect the singular locus $\Sigma$ of $M_{(0,2H,-4)}$.

Both the components $M_1$ and $M_2$ have the structure of a $\PP^3$-bundle over a smooth moduli space of dimension $4$. Let us recall these structures.

If $F=i_*G\in M_1$, then $i_*\det G\in M_{(0,H,-3)}$; this gives a morphism 
$$m_1\colon M_1\to M_{(0,H,-3)},$$ whose fibres are $\PP^3$ by \cite[Theorem~2]{NarashimanRamanan:ModuliOfVectorBundles} (cf.\ \cite[Proposition~3.7.4]{DeCataldoRapagnetta:HodgeNumbersOG10}). 

Now a stable sheaf $F$ supported on a non-reduced curve $C=2C'$ fits in a short exact sequence (\cite[Lemma~3.3.1]{Mozgovyy})
$$0\longrightarrow E\otimes I\longrightarrow F\longrightarrow E\longrightarrow0,$$
where $I$ is the ideal of $C'$ in $C$ and $E$ is the restriction of $L$ to $C'$. Moreover, $E$ is a line bundle of degree $0$ on $C'$ and $I$ has degree $-2$ on $C'$. There exists a well defined map $m\colon M_2^0\to M_{(0,H,-1)}$ that sends $F$ to $E$ (\cite[Lemma~3.3.3]{Mozgovyy}). 
There is a short exact sequence (\cite[Corollary~3.2.2]{Mozgovyy})
$$0\to\operatorname{Ext}^1_{\cO_{C'}}(E,E\otimes I)\to\operatorname{Ext}^1_{\cO_C}(E,E\otimes I)\to\operatorname{End}_{\cO_{C'}}(E\otimes I)\to0,$$
so that the fibre of $m$ over a point $E\in M_{(0,H,-1)}$ is identified with the affine space
$$\PP\operatorname{Ext}^1_{\cO_C}(E,E\otimes I)\setminus\PP\operatorname{Ext}^1_{\cO_{C'}}(E,E\otimes I)\cong\operatorname{Ext}^1_{\cO_{C'}}(E,E\otimes I)\cong\oH^1(I)=\CC^3.$$
It follows that $m\colon M_2^0\to M_{(0,H,-1)}$ is a $\CC^3$-bundle. 
Moreover, if $F\in\PP\operatorname{Ext}^1_{\cO_{C}}(E,E\otimes I)$, then $F\in M_2$, so that the natural extension 
$m_2\colon M_2\to M_{(0,H,-1)}$ is a $\PP^3$-bundle whose fibre over $E$ is identified with the projective space $\PP\operatorname{Ext}^1_{\cO_C}(E,E\otimes I)$. 

In particular we see that these two $\PP^3$-bundles intersect along each fibre in a $\PP^2$. (More precisely, we can identify this $\PP^2$ with the locus of extensions $\PP\operatorname{Ext}^1_{\cO_{C'}}(E,E\otimes I)$.)

Finally, we recall the following fact: if $L\in M_{(0,H,-3)}$, then the intersection of $m_1^{-1}(L)$ with the singular locus is isomorphic to the Kummer surface associated to jacobian of the genus $2$ curve, support of $L$ (\cite{NarashimanRamanan:ModuliOfVectorBundles}).

Passing to the symplectic resolution $\widetilde{M}_{(0,2H,-4)}$, we denote by $\widetilde{M_1}$ and $\widetilde{M_2}$ the corresponding (generic) $\PP^3$-bundles.

Our task is to determine the wall divisors associated to these two $\PP^3$-bundles. Since a wall divisor is dual to the extremal ray of the Mori cone governing the associated birational transformation, we see that in fact there is just one wall divisor, and that these two $\PP^3$-bundles must be flopped together in order to get a K\"ahler Mukai flop. In fact, the curve class associated to the Mukai flop along one of the two $\PP^3$-bundles is just the class of a line in a fibre. Since they intersect along each fibre in a $\PP^2$, it follows that such curve class must be the same.

We are now going to determine this curve class and the corresponding wall divisor. Let $R$ be a line inside a generic fibre of $M_1$; as usual, we denote by $\widetilde{R}$ the strict transform of $R$ in $\widetilde{M}_{(0,2H,-4)}$ and by $\pi\colon\widetilde{M}_{(0,2H,-4)}\to M_{(0,2H,-4)}$ the desingularisation morphism. We will determine the coefficients of $\widetilde{R}$ in $\Pic(\widetilde{M}_{(0,2H,-4)})_{\QQ}=\operatorname{span}_{\QQ}\{a,b,\sigma\}$, where $a=(-1,H,0)$ and $b=(0,0,1)$ are the generators of $v^\perp_{\operatorname{alg}}$, and $\sigma$ is the class of the exceptional divisors of $\pi$.

Since we are only interested in the class of $\widetilde{R}$, we are allowed to pick any representative in its algebraic equivalence class. For the next result, it is more convenient to think of $R$ as a line in $M_2$.
\begin{lem}
$\widetilde{R}.a=1$ and $\widetilde{R}.b=0$.
\end{lem}
\begin{proof}
By Remark~\ref{rmk:intersezione sopra e sotto}, the projection formula reduces the problem to compute $R.a$ and $R.b$. 
Taking $R\in M_2$, we can give it a modular description using a universal family of extensions of $\PP\operatorname{Ext}^1_{\cO_S}(E,E\otimes I)$, and restrict it to $R$. The Grothendieck--Riemann--Roch formula yields the claimed intersections, $R.a=1$ and $R.b=0$, as in the proof of Lemma~\ref{claim}.

Alternatively, since $R$ is contained in a fibre of the Fitting map $p\colon M_{(0,2H,-4)}\to|2H|$, it must be a multiple of $b=p^*\cO(1)$; and because both $R$ and $b$ are primitive and positive, the claim follows.
\end{proof}
Finally, we need to determine the intersection of $R$ with the singular locus $\Sigma$.
\begin{lem}
$R$ intersects $\Sigma$ in four points with multiplicity $1$.
\end{lem}
\begin{proof}
We have already remarked that $m_1^{-1}(L)\cap\Sigma$ is a Kummer surface. In particular it is a quartic surface, so that choosing $R$ generically enough it meets this surface in four smooth points.
\end{proof}
As a corollary, we get that
$$\widetilde{R}=b-\frac{2}{3}\sigma\in\Pic(\widetilde{M}_{(0,2H,-4)}),$$
and the associated wall divisor is $D=3\widetilde{R}=3b-2\sigma$, which has square $-24$ and divisibility $3$.


\subsection{Example: the hyperelliptic birational map}\label{example:hyperelliptic}
We keep the same notations as in the previous Section~\ref{subsec:example -24}. So, in particular, we work with the moduli space $M_{(0,2H,-4)}$ and we denote by $p\colon M_{(0,2H,-4)}\to|2H|$ the Fitting--support morphism. Since any smooth curve $C$ in the linear system $|2H|$ is hyperelliptic, there exists a unique $\mathfrak{g}^1_2(C)$. Tensoring with this degree $2$ line bundle defines a birational map
\begin{equation}
\varphi\colon M_{(0,2H,-4)}\dashrightarrow M_{(0,2H,-2)}.
\end{equation}
We denote by $\tilde{\varphi}\colon\widetilde{M}_{(0,2H,-4)}\dashrightarrow\widetilde{M}_{(0,2H,-2)}$ the birational map between the respective symplectic desingularisations.

These birational maps have already been considered in \cite[Section 4.1]{ono}, where it is shown that $\varphi$ does not preserve the singular loci of the respective moduli spaces.

In this section we want to understand the wall divisor associated to $\tilde{\varphi}$. The first step consists in understanding the indeterminacy locus of $\varphi$. 

It is clear that $\varphi$ is defined fibrewise, so it is enough to understand the indeterminacy locus on each fibre.
\begin{prop}
\begin{enumerate}
\item If $C$ is irreducible, then $\varphi$ is well defined;
\item If $C$ is reduced, but not irreducible then $\varphi$ is well defined only on stable sheaves;
\item If $C$ is not reduced, then $\varphi$ is not defined.
\end{enumerate}
\end{prop}
\begin{proof}
Recall that $C$ is the double cover of a conic in $\PP^2$, ramified along a sextic curve; the pullback of a point on the conic defines a $\mathfrak{g}^1_2$ on $C$.
\begin{enumerate}
\item If the conic is smooth, there exists only one equivalence class of such a point, so the $\mathfrak{g}^1_2$ is uniquely determined. In this case, the curve $C$ is irreducible (but possibly singular if the conic is tangent to the sextic curve) and the map $\varphi$ is well defined. \\

\item Suppose that $C=C_1\cup C_2$ is reducible; we have two cases.

\emph{Both $C_1$ and $C_2$ are smooth.} Let $F=i_*G$ be a sheaf in $M_{(0,2H,-4)}$. Each $C_j$ has a $\mathfrak{g}^1_2$, call it $\cO_{C_j}(p_j+q_j)$ with $p_j,q_j\notin C_1\cap C_2$, and we denote by $\varphi_j(F)$ the sheaf $F(p_j+q_j)$. Both $\varphi_j(F)$ belong to the moduli space $M_{(0,2H,-2)}$. We claim that $\varphi_1(F)=\varphi_2(F)$ if and only if $F$ is stable. Moreover, in this case $\varphi(F)$ is strictly semistable.

Let us denote by $G_j$ the torsion free part of the restriction of $G$ to $C_j$; $G_j$ is a line bundle on $C_j$ of degree $d_j$. With an abuse of notation we keep calling $i$ the inclusion of the curve $C_j$ in $S$. Since $F$ is semistable, we have that $d_j\geq-1$.
There is a short exact sequence
$$0\to F\to i_*G_1\oplus i_*G_2\to Q\to0,$$
where $Q$ is a torsion sheaf supported on $C_1\cap C_2$. If $F$ is stable, the only possibility is that $Q$ has length $2$ and $d_1=d_2=0$. If the length of $Q$ is $2$ and $F$ is strictly semistable, then either $d_1=1$ and $d_2=-1$ or $d_1=-1$ and $d_2=1$. If the length of $Q$ is $1$, then $F$ is necessary strictly semistable and either $d_1=0$ and $d_2=-1$ or $d_1=-1$ and $d_2=0$. Finally if $Q=0$, then $F$ is polystable and $d_1=d_2=-1$.

Let us now study the semistability of $\varphi_1(F)$. As for $F$, we do that by studying the torsion free restrictions $\varphi_1(F)_1=i_*(G_1(p_1+q_1))$ and $\varphi_1(F)_2=i_*(G_2)$, of degree (respectively) $d_1+2$ and  $d_2$. 

Suppose that $F$ is stable, in which case $d_1=d_2=0$. It follows that the S-equivalence class of $\varphi_1(F)$ is $[i_{*}G_1(p_1+q_1-n_1-n_2)\oplus i_{*}G_2]=[i_{*}G_1\oplus i_{*}G_2]$, since $p_1+q_1-n_1-n_2$ is linearly equivalent to zero on $C_1$. Here $n_1$ and $n_2$ are the nodes of $C$ (possibly coinciding -- recall that $Q$ has length $2$ in this case). 

If $F$ is strictly semistable, then we have three cases, depending on the length of $Q$. If the length of $Q$ is $2$, then either $d_1+2=1$ and $d_2=1$ or $d_1+2=3$ and $d_2=-1$; it follows that $\varphi_1(F)$ is stable in the first case and unstable in the second case. If the length of $Q$ is $1$, then either $d_1+2=2$ and $d_2=-1$ or $d_1+2=1$ and $d_2=0$; it follows that $\varphi_1(F)$ is unstable in the first case and strictly semistable in the second case. If $Q=0$, then $d_1+2=1$ and $d_2=-1$, so $\varphi_1(F)$ is unstable.

We can perform the same analysis for $\varphi_2(F)$ and confront the result with the previous case. For example, if $F$ is stable, then $\varphi_2(F)$ is strictly semistable and its S-equivalence class is $[i_{*}G_1\oplus i_{*}G_2(p_2+q_2-n_1-n_2)]=[i_{*}G_1\oplus i_{*}G_2]$, since $p_2+q_2-n_1-n_2$ is linearly equivalent to zero on $C_2$. Again $n_1$ and $n_2$ are the nodes of $C$. On the other hand, if $F$ is strictly semistable and $Q$ has length $1$, then $\varphi_2(F)$ is stable, but it is not isomorphic to $\varphi_1(F)$. Finally, there is a remaining case (when $F$ is strictly semistable and $Q$ has length $1$) where both $\varphi_1(F)$ and $\varphi_2(F)$ are strictly semistable, but their S-equivalence classes are different.
The claim follows.
 
\emph{At least one between $C_1$ and $C_2$ is singular.} The main tool we used for the smooth case is the Riemann--Roch formula, which holds in the singular case for locally free sheaves (e.g\ \cite[Exercise~IV.1.9]{Hartshorne:AG}). The same computations can be performed in the singular case using the following remark. Sheaves supported on a singular curve $C$ are of the form $i_*F$, where $F$ is torsion free of rank $1$. We then have two cases, either $F$ is locally free or $F$ is the pushforward of a line bundle on a (partial) normalisation. \\

\item Now the curve $C=2C'$ is not reduced; the reduced curve $C'$ is either smooth or has at worst nodal or cusp singularities. It is clear that $C$ itself has not a $\mathfrak{g}^1_2$, so $\varphi$ is not defined on the irreducible component $M_2$ (see Section~\ref{subsec:example -24}). On the other hand, if $F=i'_*E$, where $E$ is a vector bundle of rank $2$ on $C'$, then tensoring by the $\mathfrak{g}^1_2$ of $C'$ gives a vector bundle of degree $2$, which does not belong to $M_{(0,2H,-2)}$, so $\varphi$ cannot be defined on this locus either.
\end{enumerate}
\end{proof}

\begin{cor}\label{cor:indet_phi}
The indeterminacy locus of $\varphi$ consists of the union of the singular locus of $M_{(0,2H,-4)}$ and two $\PP^3$-bundles. 
\end{cor}
\begin{proof}
We already determined the indeterminacy locus, and we saw that it consists of the singular locus union the space $M_\Delta=M_1\cup M_2$ of sheaves whose Fitting support is non-reduced. We have already seen that both $M_1$ and $M_2$ are $\PP^3$ bundles over four dimensional irreducible holomorphic symplectic manifolds.
\end{proof}

\begin{cor}
The indeterminacy locus of the map $\widetilde{\varphi}$ is the union of a $\PP^3$-bundle and a generic $\PP^3$-bundle.
\end{cor}
\begin{proof}
By Corollary~\ref{cor:indet_phi}, $\varphi$ is not well defined on the singular locus and on the $\PP^3$-bundles $M_1$ and $M_2$. The resolution of singularities is a blow-up, therefore the map $\widetilde{\varphi}$ extends to the generic point of the exceptional divisor, and therefore also on every point where the blow-up map has a one dimensional fibre (i.e.\ where the singularity is of type $A_1$). If $\Sigma$ is the singular locus, then the singularities of $M_{(0,2H,-4)}$ are of type $A_1$ exactly at the locus $\Sigma\setminus(\Sigma\cap M_\Delta)$. The claim follows.
\end{proof}
The corollary implies that the indeterminacy locus of $\widetilde{\varphi}$ coincides with the locus described in Section~\ref{subsec:example -24}. In particular the wall divisor associated to it has degree $-24$ and divisibility $3$, and $\widetilde{\varphi}$ is the Mukai flop around the union of two $\PP^3$-bundles.

\section{The K\"ahler cone}\label{sec:ample}

In this section we want to give a numerical description of wall divisors, and therefore of the ample (and K\"ahler) cone. 
Our strategy moves in two directions: on one hand, we will produce Wall divisors by looking at special extremal contractions and on the other hand we will prove that some monodromy orbits are not associated to Wall divisors by giving ample classes orthogonal to special elements in the orbit. Both steps will involve in a fundamental way the knowledge of the birational geometry of the singular moduli spaces of O'Grady type.

\begin{prop}\label{prop:murisotto}
Let $X$ be a tenfold of O'Grady type and let $D$ be a wall divisor on $X$. Then there exist a projective K3 surface $S$, $v=2w\in H^{2*}(S)$ a Mukai vector with $v^2=8$ and $H$ a $v$ generic polarization such that the parallel transport from $X$ to $\widetilde{M}_v(S,H)$ sends $D$ to a wall divisor $D'$ and the projection of $D'$ inside $\widetilde{\Sigma}^\perp=H^2(M_v(S,H))$ is a multiple of one of the following:
\begin{itemize}
\item A divisor of square $-2$ and divisibility $1$,
\item A divisor of square $-2$ and divisibility $2$,
\item A divisor of square $-4$ and divisibility $1$,
\item A divisor of square $-10$ and divisibility $2$.
\end{itemize}
\end{prop}
\begin{proof}
We want to deform $X$ to the resolution of a moduli space of sheaves on a K3 surface in such a way that the parallel transport $D'$ of $D$ and $\widetilde{\Sigma}$ generate a negative definite lattice. Without loss of generality, we can take $v=(2,0,-2)$, see Lemma \ref{lem:eichler_gen} and $D'=aE+k\widetilde{\Sigma}$, with $E\in \widetilde{\Sigma}^\perp$ and $a,k\in\frac{1}{2}\mathbb{Z}$ with $a+k\in\mathbb{Z}$ (see example \ref{ex:M_S}). If $\div(D)=1$, we can take $k=0$ and $E^2<0$ by Lemma \ref{lem:eichler}, otherwise we have $k=1$ or $k=1/2$ and $E^2<0$. Let $T$ be the saturation of the lattice generated by $D'$ and $\widetilde{\Sigma}$ (which coincides with the saturation of the lattice generated by $E$ and $\widetilde{\Sigma}$ and is therefore negative definite). Up to replacing $D'$ with an isometric element, we can suppose that $D'^\perp\cap \overline{\mathcal{BK}}_{\widetilde{M}_v(S,H)}\neq 0$ by the Kawamata Morrison cone conjecture for the Movable cone (see \cite[Section 6]{mark_tor}).
This implies that there is an extremal curve $R$ on a (possibly different) smooth birational model of $\widetilde{M}_v(S,H)$ which, up to embedding $H_2(X,\mathbb{Z})$ in $H^2(X,\mathbb{Q})$ by lattice duality, lies in $T\otimes\mathbb{Q}$ and is not a multiple of $\widetilde{\Sigma}$. Notice that, on every smooth birational model of $\widetilde{M}_v(S,H)$, this extremal curve is either effective or anti-effective by \cite{bht} and \cite{Mongardi:Cones}.
 As our manifold is projective, we can take a very general big and movable divisor $P$ in ${D'}^\perp \cap \widetilde{\Sigma}^\perp$ (the orthogonal to $T$ has signature $(1,\rho-3)$ in $NS(\widetilde{M}_v(S,H))$). Let $P':=P-\epsilon D'-\eta \widetilde{\Sigma}$ for small $\epsilon$ and $\eta$: clearly, it is still a big divisor. 
Up to replacing $P'$ with a multiple, we can suppose that $P'$ is big and effective. We can also take $P$ (and therefore also $P'$ and the surface $S$) very general with respect to these properties, i.e. all negative divisors orthogonal to $P$ are in $T$ (and the Picard rank of $S$ is one). Let us run the MMP for the pair $(\widetilde{M}_v(S,H),\mu P')$ (for $\mu$ small enough, the pair is klt, see \cite[Remark~12]{hass_ts_moving}). By \cite[Theorem 4.1]{lp}, we can contract the divisor $\widetilde{\Sigma}$ as the first step of this MMP and it will nevertheless terminate. Therefore, at the second step we are performing the MMP for the pair $(M_v(S,H),\mu \pi(P'))$. As the class $P$ was in $T^\perp\subset \widetilde{\Sigma}^\perp$, we can consider the ($\mathbb{Q}$-Cartier) divisor $\pi(P)$ identified with $P$ under the identification $H^2(M_v(S,H))\cong \widetilde{\Sigma}^\perp$. It follows that $\pi(P')=P-\epsilon \pi(D')$ under this identification. As $\pi(P)$ is still orthogonal to a single negative class given by $T\cap\widetilde{\Sigma}^\perp$, the MMP for the pair $(M_v(S,H),\mu \pi(P'))$ terminates after at most one step. As either $\pi(R)$ or $\pi(-R)$ is an effective curve negative with $\pi(P')$, there is at least one step in this MMP which contracts it or flops it. Therefore, by Theorem \ref{thm:mz2}, there is a wall in the space of stability conditions given by a negative class orthogonal to $\pi(P)$. As the only such class is a generator of $T\cap\widetilde{\Sigma}^\perp$ the claim follows by the classification of these classes contained in Theorem \ref{thm:mz1}.  
\end{proof}

\begin{oss}\label{oss:murisotto}
The statement of Proposition \ref{prop:murisotto} actually holds for all very general K3 surfaces such that $\langle \widetilde{\Sigma},D' \rangle$ is negative definite. Indeed, in this case there is a MMP contracting or flopping all the curves in this negative lattice, which in turn implies that the singular moduli space has some non generic stability conditions. 
\end{oss}
\begin{prop}\label{wall:esclusione_a_raffica}
Let $X$ be a manifold of O'Grady type and let $D\in \Pic(X)$ be a wall divisor. Then one of the following is satisfied:
\begin{enumerate}
\item $D^2=-2$ and $\div(D)=1$,
\item $D^2=-4$ and $\div(D)=1$,
\item $D^2=-6$ and $\div(D)=3$,
\item $D^2=-24$ and $\div(D)=3$.
\end{enumerate}
\end{prop}
\begin{proof}
Up to the monodromy action, we can suppose that $X\cong \widetilde{M}_v(S,H)$ for a projective K3 surface $S$, a Mukai vector $v=(2,0,-2)$ of square $8$ and a $v$-generic polarization $H$. Let us first suppose that $\div(D)=1$.  Up to the monodromy action, we can suppose $q(D,\widetilde{\Sigma})=0$ by Lemma \ref{lem:eichler}. Applying Proposition \ref{prop:murisotto}, this immediately implies $D^2=-2$ or $-4$ as claimed.
Let us now suppose $\div(D)=3$. This time, up to the monodromy action, we can write $D=3E+\widetilde{\Sigma}$ by Lemma \ref{lem:eichler}, with $E\in \widetilde{\Sigma}^\perp$ primitive (and actually, inside $H^2(S)$). Therefore, the projection of $D$ in $\widetilde{\Sigma}$ is $E$. Proposition \ref{prop:murisotto} now implies that one of the following holds: $E=0,$ $E^2=-2$ or $E^2=-4$. The first two cases correspond to the cases in our claim, let us exclude the third one.\\
Let us take a very general polarized K3 surface $(S,H)$ of degree $2$ and take the Mukai vector $(0,2H,-2)$ (and $H$ as the $v$-generic polarization). Let us take the two orthogonal generators $D_1=(-2,H,0)$ and $D_2=(-2,H,-1)$ of $\Pic(M_v(S,H))$. Notice that $D_2$ has divisibility $2$ in the singular moduli space. Let $X:=\widetilde{M}_v(S,H)$. We have $\Pic(X)=\langle D_1,\frac{D_2+\widetilde{\Sigma}}{2},\widetilde{\Sigma} \rangle$, where $\widetilde{\Sigma}$ is the exceptional divisor. Let $W=4D_1+5D_2$, it has square $-18$ and $\div(W)=2$ in the singular moduli space. Notice that $\frac{W-\widetilde{\Sigma}}{2}\in\Pic(X)$. Suppose now that divisors of square $-42$ and divisibility $3$ are wall divisors. It would follow that $W':=3\frac{W-\widetilde{\Sigma}}{2}+\widetilde{\Sigma} $ is a wall divisor such that $\langle W',\widetilde{\Sigma} \rangle$ is negative definite, hence the projection of $W'$ inside $\widetilde{\Sigma}^\perp$ must be a multiple of one of the cases contained in Theorem \ref{thm:mz1} by Remark \ref{oss:murisotto} and the proof of Proposition \ref{prop:murisotto}. However, this projection is clearly a multiple of $W$, which gives the desired contradiction.
\end{proof}

\begin{prop}\label{prop:wall}
Let $X$ be a manifold of OG10 type and $D\in \Pic(X)$. Then $D$ is a wall divisor if and only if one of the following holds:
\begin{enumerate}
\item $D^2=-2$ and $\div(D)=1$,
\item $D^2=-6$ and $\div(D)=3$,
\item $D^2=-4$ and $\div(D)=1$,
\item $D^2=-24$ and $\div(D)=3$.
\end{enumerate}
Moreover, there is a curve of class proportional to $D$ covering a divisor in the first two cases, a codimension $5$ rational subvariety in the third case and a codimension $3$ subvariety in the last case.
\end{prop}
\begin{proof}
By Proposition~\ref{wall:esclusione_a_raffica}, the above cases are the only ones which can give wall divisors. By Proposition~\ref{prop:pex}, the first two cases have a multiple which is stably prime exceptional, hence they are wall divisors and there is a curve ruling a divisor proportional to them by \cite[Section 5 and 6]{mark_tor}.
Notice that if a wall divisor corresponds to a codimension $i$ contraction on some manifold, by \cite[Theorem 1.6]{amver2} the rational contracted curve deforms in the Hodge locus of the wall divisor and always covers a codimension $i$ subvariety and these subvarieties are always diffeomorphic.
By Example~\ref{example:lP^5} and Example~\ref{example:zero section}, the third case is a wall divisor corresponding to a lagrangian $\PP^5$. Finally the last case follows from Example~\ref{example:lP^3-bundle}.
\end{proof}
\begin{thm}\label{thm:ample}
Let $X$ be a manifold of OG10 type. Then, the K\"ahler cone of $X$ is one of the connected components of $$\mathcal{C}(X)\setminus \bigcup_{(0>D^2\geq -4) \text{ or } (\div(D)=3 \text{ and } 0>D^2\geq -24)} D^\perp.$$

\end{thm}
\begin{proof}
Notice that if $D$ is a divisor of non negative square, then $D^\perp\cap \mathcal{C}(X)=\emptyset$, so the relevant divisors for the claim are only those of negative square. 
By \cite{Mongardi:Cones}, the K\"ahler cone is a connected component of $$\mathcal{C}(X)\setminus \bigcup_{D \text{ wall divisor }} D^\perp.$$
By Proposition \ref{prop:wall}, wall divisors are divisors of square either $-2$ or $-4$ and divisors of divisibility $3$ and square either $-6$ or $-24$, therefore the claim follows. 

\end{proof}

\section{Countering the counterexample: two walls coming together}\label{sec:counter-counter}

The previous sections, especially Example~\ref{example:lP^5} and Example~\ref{example:lP^3-bundle}, show that there are birational transformations of the smooth O'Grady tenfolds which are not well visible in the singular moduli space.

This phenomenon is what led to the wrong statement in \cite{monwrong} that the monodromy group of manifolds of OG10 type is strictly contained in the group of orientation preserving isometries, which is in clear contradiction with \cite{ono}. This section is meant to serve as an erratum to \cite{monwrong}, and we try to explain what goes wrong, why and where.

Let us look more closely at the (wrong) counterexample contained in \cite{monwrong}: the author claims that a divisor $D$ of square $-10$ and divisibility 2 on the singular moduli space $M_v(S,H)$ (with $H$ $v$-generic) pulls back to a wall divisor $\widetilde{D}$ on the resolution $\widetilde{M}_{v}(S,H)$, where they are divisors of square $-10$ and divisibility one. Then, by proving that there are ample classes orthogonal to some divisors of square $-10$ in an appropriate projective deformation of $\widetilde{M}_{v}(S,H)$, he proves that the monodromy group cannot be maximal, as otherwise it would send a wall divisor into something orthogonal to an ample class.

What goes wrong in this argument is the first claim (and the others are correct, although the example contained in \cite[Example 5.2]{monwrong} does not work as claimed, see Remark \ref{oss:horrorsymmetriae}). 
Indeed, the pull back of an element of square $-10$ and divisibility $2$ on $M_v(S,H)$ need not be a wall divisor, however there are two contractions associated with it, in the sense that the saturated lattice containing $\widetilde{D}$ and $\widetilde{\Sigma}$ is generated (over $\mathbb{Q}$) by wall divisors, essentially as is done in Proposition~\ref{prop:murisotto}.
\begin{prop}
Let $(S,H)$ be a very general degree two K3 surface and let $v=(2,0,-2)$ be a Mukai vector. Let $D=(3,2H,3)\in H^2(M_v(S,H))$. Then, the saturated lattice $T\subset H^2(\widetilde{M}_v(S,H))$ containing the pullback $\pi^*D$ of $D$ and the exceptional divisor $\widetilde{\Sigma}$ is generated by $\widetilde{\Sigma}$ and a divisor $E$ with $E^2=-4$ and $(E,\widetilde{\Sigma})=3$. Moreover, $E$ and $3E+\Sigma$ are both wall divisors corresponding to small contractions.
\end{prop}
\begin{proof}
With an abuse of notation, we put $\pi^*D=D$.
By Theorem \ref{thm:og10prelim}, the class $D-\widetilde{\Sigma}$ is twice an integral class $E$. As $D$ and $\widetilde{\Sigma}$ are orthogonal, we have $E^2=-4$ (and therefore clearly $\div(E)=1$), which is our first claim. The second claim follows directly from Proposition \ref{prop:wall}, as $E^2=-4$ and $3E+\Sigma$ has square $-24$ and divisibility $3$. Notice that the curve giving the $\mathbb{P}^3$ bundle in the singular moduli space pulls back to the curve dual to $3E+\Sigma$, as in Example~\ref{example:lP^3-bundle}. However, the $-4$ curve dual to $E$ cannot be clearly seen in the singular moduli space, but from Example~\ref{example:lP^5} we know that it is linked to a lagrangian $\mathbb{P}^5$. 
\end{proof}
As both $E$ and $3E+\Sigma$ are wall divisors, \cite{monwrong} does not give a contradiction. 
The second ingredient in the proof of \cite{monwrong} was an induced automorphism on the resolution $\widetilde{M}_S(H)$ of $M_S(H)$, where $S$ has a non trivial finite order symplectic automorphism $\varphi$ and $H$ is $\varphi$ invariant. This is also an issue, because there is NO $v$-generic $\varphi$ invariant polarization and NO finite order symplectic automorphism on a K3 inducing a regular automorphism on a manifold of OG10-type, as is also happening for OG6-type manifolds (see \cite{gov}): 

\begin{oss}\label{oss:horrorsymmetriae}
Let $S$ be a K3 surface and let $\varphi\in \mathrm{Aut}(S)$ be a finite order non trivial symplectic automorphism. Let $v$ be a $\varphi$ invariant Mukai vector with $v=2w$ and $v^2=8$ and let $H$ be an invariant polarization. Then, the natural action of $\varphi$ on $M_v(S,H)$ does not extend to the minimal resolution $\widetilde{M}_v(S,H)$. Indeed, if it extended to an automorphism, $\widetilde{M}_v(S,H)$ would have an invariant ample class. However, the action of $\varphi$ on $H^2(\widetilde{M}_v(S,H))$ is given by its action on $v^\perp$ inside $H^{even}(S)$, so that the coinvariant part is determined in \cite{hashi} for any $\varphi$ (actually, any finite group). This coinvariant lattice is orthogonal to any invariant class and always contains elements of square $-4$ by \cite{hashi}, which are wall divisors by Proposition \ref{prop:wall}, hence there are no invariant ample classes.
\end{oss}
More precisely, the automorphisms are not well defined exactly along the $\mathbb{P}^5$ linked with a class of square $-4$, as in Example~\ref{example:lP^5} (see Remark~\ref{rmk:lP^5 deformato}).\\
This last issue can be circumvented easily, as it is possible to find different examples where a class of square $-10$ is orthogonal to an ample class.

\section{Examples and consequences for intermediate jacobians}\label{sec:IJac}


Let $V\subset\PP^5$ be a general cubic fourfold. The smooth linear sections $Y\subset V$ have a principally polarised intermediate Jacobian, giving rise to an intermediate Jacobian fibration $\cJ_V$. We denote by $\cJ^k_V$ the torsor parametrising degree $k$ cycles on the (smooth) linear sections of $V$. Notice that, up to canonical isomorphism, there exists only two intermediate Jacobians, namely $\cJ_V=\cJ^0_V$ and $\cJ^1_V=\cJ^2_V$. By \cite{LSV} and \cite{VoisinLSV}, there exist smooth and symplectic compactifications $\IJac(V)$ and $\IJac^t(V)$ of $\cJ_V$ and $\cJ^1_V$, respectively.  Both $\IJac(V)$ and $\IJac^t(V)$ are irreducible holomorphic symplectic manifolds of OG10 type. Moreover, they both come with the structure of a lagrangian fibration over $\PP^5$. In the following we denote by $p\colon\IJac(V)\to\PP^5$ and $p^t\colon\IJac^t(V)\to\PP^5$ these fibrations, and by $b_V:=p^*\mathcal{O}(1)\in\Pic(\IJac(V))$ and $b_V^t:=p^{t*}\mathcal{O}(1)\in\Pic(\IJac^t(V))$ the corresponding classes. 

Let $\mathcal{F}_V$ be the relative Fano variety of lines: it exists as a projective variety of dimension $7$, fibred over $\PP^5$. We notice that $\mathcal{F}$ is a $\PP^3$-bundle over $F(V)$, the Fano variety of lines of $V$.

There are two natural rational maps. The first one is
\begin{equation}\label{eqn:Theta}
\delta\colon\mathcal{F}_V\times_V\mathcal{F}_V\dashrightarrow\cJ_V,
\end{equation}
defined on the locus fibred over smooth linear sections by sending two lines to the Abel--Jacobi invariant of their difference (cf.\ \cite{ClemensGriffiths}). The closure in $\IJac(V)$ of ts image (with reduced scheme structure) is a relative theta divisor, which we denote by $T_V$. The second one is
\begin{equation}\label{eqn:twisted Theta}
s\colon\mathcal{F}_V\times_V\mathcal{F}_V\dashrightarrow\cJ^t_V,
\end{equation} 
defined on the locus fibred over smooth linear sections by sending two lines to the (twisted) Abel--Jacobi invariant of their sum. 
By \cite[Th\'eor\`eme~1.4]{Druel:ModuliSpaces}, this map is generically finite of degree $2$, hence the closure in $\IJac^t(V)$ of its image (with reduced scheme structure) defines a divisor $T^t_V$ in $\IJac^t(V)$, called relative twisted theta divisor.

Define the sublattices $P_V:=\langle T_V,b_V\rangle\subset\Pic(\IJac(V))$ and $P^t_V:=\langle T^t_V,b^t_V\rangle\subset\Pic(\IJac^t(V))$. 

\begin{lem}\label{prop:primitive_cohom}
Let $V$ be a general cubic fourfold. 
\begin{enumerate}
\item \begin{equation*}
P_V=\left(
\begin{array}{cc}
-2 & 1 \\
1 & 0
\end{array}\right)
\qquad\mbox{ and }\qquad
P^t_V=\left(
\begin{array}{cc}
-18 & 3 \\
3 & 0
\end{array}\right).
\end{equation*}

\item There exist isomorphisms of Hodge structures
\begin{equation*}
\alpha\colon\oH^4(V,\ZZ)_{\operatorname{prim}}\to P_V^\perp\qquad\mbox{ and }\qquad\alpha^t\colon\oH^4(V,\ZZ)_{\operatorname{prim}}\to (P_V^t)^\perp,
\end{equation*}
where the orthogonal complements are taken in $\oH^2(\IJac(V),\ZZ)$ and $\oH^2(\IJac^t(V),\ZZ)$, respectively. Moreover, $\alpha$ and $\alpha^t$ are anti-similitudes, that is they preserve the symmetric bilinear forms\footnote{$\oH^{2,2}(V,\ZZ)_{\operatorname{prim}}$ is endowed with the non-degenerate bilinear form induced by intersection product; while $P_V$ (resp.\ $P_V^t$) is endowed with the non-degenerate bilinear form induced by the Beauville--Bogomolov--Fujiki form on $\oH^2(\IJac(V),\ZZ)$ (resp.\ $\oH^2(\IJac^t(V),\ZZ)$).} up to a negative constant.

\item If $V$ is very general (in the sense of Hassett), then 
$$\Pic(\IJac(V))=P\qquad\mbox{ and }\qquad\Pic(\IJac^t(V))=P^t.$$
\end{enumerate}
\end{lem}
\begin{proof}
The claim about the non-twisted case is \cite[Proposition~4.1]{ono}. 
The twisted case follows directly from the non-twisted one: let us outline the main differences\footnote{The same construction already appeared in \cite[Section~4.5]{onoPhD}, but there it was an initial step to prove a wrong result, namely the existence of a twisted theta divisor of divisibility $1$. However, the argument works in this situation and we rephrase it here for better clarity.}.
Deforming $V$ to a Pfaffian cubic fourfold $V_0$, as explained in \cite[Example~4.3.6]{onoPhD}, the two varieties $\IJac(V_0)$ and $\IJac^t(V_0)$ are isomorphic. Moreover, in this case $T_{V_0}^t=3T_{V_0}$, while $b_{V_0}^t=b_{V_0}$; item (1) follows then from the non-twisted case.

Now, using the surjectivity of the twisted Abel--Jacobi map for cubic threefolds, and the universal property of the twisted intermediate jacobians (cf.\ \cite[Proposition~3.1]{VoisinLSV}), the proof of \cite[Lemma~1.1]{LSV} produces a rational cycle $T_0\in\operatorname{CH}^2(\mathcal{J}_V\times V)_{\QQ}$; we denote by $T\in\operatorname{CH}^2(\IJac^t(V)\times V)_{\QQ}$ its closure. The map $\alpha^t$ is defined as follows. First of all, let $\mathcal{Y}_V$ be the universal family of linear sections of $V$, and let $q\colon\mathcal{Y}_V\to V$ be the morphism that is the inclusion of each linear section (notice that $q$ is a projective bundle). The composition
$$ T^*\circ q^*\colon\oH^4(V,\QQ)\longrightarrow\oH^2(\IJac^t(V),\QQ)$$
is a morphism of Hodge structures. Moreover, if $h$ is the hyperplane section of $V$, then the class $h^2$ is sent to $P^t\otimes\QQ$. The map $\alpha^t$ is then the restriction of $T^*\circ q^*$ to the primitive cohomology of $V$. The proof now goes exactly as the proof of \cite[Proposition~4.1]{ono}.

Finally, if $V$ is very general, then $\Pic(\IJac^t(V))$ is an overlattice of $P^t$. Notice that $P_V^t\cong3U$, where $U$ is the unimodular hyperbolic plane. By \cite{PertusiCo}, $\IJac^t(V)$ can be realised as a symplectic desingularisation of a singular moduli space of Bridgeland-semistable objects on the K3 category of $V$, and the exceptional divisor of this desingularisation has divisibility $3$. It follows that $\Pic(\IJac^t(V))$ contains elements of divisibility $3$, hence it must coincide with $P_V^t$.
\end{proof}
The following is a straightforward generalisation of the result above. Let $V$ be a smooth cubic fourfold and $X$ any smooth and symplectic compactification of $\cJ_V$, such that $X$ is a lagrangian fibration. As before, we denote by $b_X$ the class of the fibration, that is the pullback of the hyperplane class on the base. Since $X$ contains an open subset isomorphic to a relative jacobian fibration, there exists a relative theta divisor $T_X$ obtained by closing in $X$ the image (with reduced scheme structure) of the morphism (\ref{eqn:Theta}). Put $P_X:=\langle T_X,b_X\rangle\subset\Pic(X)$.
\begin{lem}\label{lemma:prim cohom general}
Keep notations as above.
\begin{enumerate}
\item $P_X$ is isometric to a unimodular hyperbolic plane.

\item There exists an isomorphism of Hodge structures
$$\alpha\colon\oH^4(V,\ZZ)_{\operatorname{prim}}\to P_V^\perp$$
that is an anti-similitude.
\end{enumerate}
\end{lem}
\begin{proof}
First of all, if $V_0$ is a smooth cubic fourfold such that $\IJac(V_0)$ exists, then by construction $X$ is birational to $\IJac(V_0)$, and the relative theta divisors are birational as well. In particular $\oH^2(X,\ZZ)\cong\oH^2(\IJac(V_0),\ZZ)$ and $P_X\cong P_{V_0}$, so that the result holds for $X$. 

For the general case, we can deform $V$ to $V_0$ such that the induced deformation from $X$ to $\IJac(V_0)$ preserves the relative theta divisors and the class of the fibrations. In particular, the lattices $P_{V_0}$ and $P_X$ are flat with respect to the Gauss--Manin connection of the corresponding deformation family, therefore $P_X$ is a unimodular hyperbolic plane.

The second claim follows exactly as in the proof of Lemma~\ref{prop:primitive_cohom}. In fact, the existence of a rational cycle $T_0\in\operatorname{CH}^2(\cJ_V\times V)_{\QQ}$ is assured by \cite[Lemma~1.1]{LSV} again, and then we close it in $X$ to get a cycle $T\in\operatorname{CH}^2(X\times V)_{\QQ}$. The morphism $\alpha$ is constructed in the same way, and the claim follows by using the fact that $P_X$ is a hyperbolic plane.
\end{proof}

\subsection{Wall divisors and geometry of $\IJac(V)$}
Let $V$ be a general cubic fourfold. 
The lagrangian fibration $\IJac(V)\to\PP^5$ has a rational section, so that there is a subvariety $Z_V\subset\IJac(V)$ birational to $\PP^5$. 

We claim that the wall divisor associated to $Z_V$ has degree $-4$ and divisibility $1$. In fact, without loss of generality we can suppose that $V$ is very general, so that $\Pic(\IJac(V))=\langle T_V,b_V\rangle$ by Proposition~\ref{prop:primitive_cohom}. Now, the divisor $T_V$ is prime exceptional by \cite[Theorem~2]{Sacca:boh}, and $b_V$ is nef. By Proposition~\ref{prop:wall}, the only wall divisor determining the K\"ahler cone must be $T_V-b_V$, which has degree $-4$, as claimed. In Figure~\ref{pic:B} below we display the K\"ahler and birational K\"ahler cone of $\IJac(V)$, when $V$ is very general. 

\begin{figure}[!ht]
\begin{center}
\begin{tikzpicture}
\tikzset{vertex/.style = {shape=circle,draw,minimum size=2em}}
\tikzset{>=latex}
\draw [thick, ->] (-0.5,0) -- (2.75,0) node[anchor=north] {$b_V$};
\draw [thick, ->] (0,-0.5) -- (0,1.75) node[anchor=north east] {$T_V+b_V$};
\draw [-] (0,0) -- (2,0.75) node[anchor=south west] {$D^\perp$};
\draw [dashed,-] (0,0) -- (1.5,1.5) node[anchor=south west] {$T_V^\perp$};
\end{tikzpicture} 
\end{center}
\caption{Picture of the K\"ahler cone of $\IJac(V)$; here $D=T_V-b_V$.}
\label{pic:B}
\end{figure}
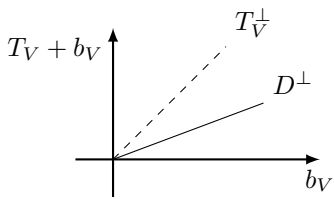​

\begin{oss}
Degenerating the cubic fourfold to a chordal cubic fourfold, the intermediate jacobian fibration degenerates to the smooth moduli space $\widetilde{M}_{(0,2H,-4)}(S)$, where $S$ is the K3 surface of genus $2$ associated to a chordal cubic fourfold (see \cite[Section~5]{KLSV} or \cite[Section~4.1.1]{ono}). The zero section $Z_V$ degenerates to the zero section $Z_S$ of Example~\ref{example:zero section}. In particular, the wall divisor associated to $Z_V$ is sent to the wall divisor associated to $Z_S$. More precisely, as remarked in \cite[Section~4.1.4]{ono}, there is a parallel transport operator 
$$ P\colon\oH^2(\IJac(V),\ZZ)\longrightarrow\oH^2(\widetilde{M}_{(0,2H,-4)}(S),\ZZ)$$
such that $P(T_V-b_V)=a-3b$ (in the notations of Example~\ref{example:zero section}).
\end{oss}

\subsection{Wall divisors and geometry of $\IJac^t(V)$}
Let $V$ be a general cubic fourfold.
There is a natural rational morphism
\begin{equation}\label{eqn:Zt}
\varphi^t\colon\mathcal{F}_V\dashrightarrow \IJac^t(V),
\end{equation}
defined by sending a line $l$ on a smooth linear section $Y$ to the Abel--Jacobi invariant of $[l]$ in $J^1_Y$. It is birational onto its image. We denote by $W^t_V$ the closure of the image of $\varphi^t$ and remark that it is generically a $\PP^3$-bundle over the Fano fourfold $F(V)$ of lines in $V$. 

We claim that the wall divisor associated to the Mukai flop around $W_V^t$ has degree $-24$ and divisibility $3$. In fact, we can suppose again that $V$ is very general, so that $\Pic(\IJac^t(V))=\langle T_V^t,b_V^t\rangle$. By Proposition~\ref{prop:pex}, the divisor $T_V^t+2b_V^t$ is prime exceptional, and by Proposition~\ref{prop:wall} the only wall divisor determining the K\"ahler cone must be $T_V^t-b_V^t$, which has degree $-24$ and divisibility $3$ as claimed. The same Figure~\ref{pic:B} above describes the K\"ahler and birational K\"ahler cone of $\IJac^t(V)$, when $V$ is very general, up to replacing $T_V$ and $b_V$ with the corresponding twisted versions $T_V^t$ and $b_V^t$, and the isotropic class $T_V+b_V$ with the corresponding isotropic class $T_V^t+3b_V^t$.


\begin{oss}
The Mukai flop around $Z_V^t$ is explicitly described in \cite[Example~6.8]{PertusiCo}, using (desingularised) moduli space of Bridgeland semistable objects on the K3-category of the cubic fourfold.
\end{oss}

\begin{oss}
There is a natural rational map
$$ t\colon\IJac^t(V)\longrightarrow\IJac(V) $$
of degree $3^{10}$. 
Denote by $W_V$ the closure of the image of $W_V^t$ under the map $t$. Then, if $V$ is very general, $W_V$ cannot be a $\PP^3$-bundle, since otherwise we should see a wall divisor of divisibility $3$ in $\Pic(\IJac(V))$.
\end{oss}


\subsection{Uniqueness of the symplectic intermediate jacobian compactification}

In this subsection, we give an answer to the following question.
\begin{ques*}
Let $V$ be a smooth cubic fourfold and let $\mathcal{J}_{U_1}(V)$ be the intermediate jacobian fibration over the open set $U_1$ of (at most) nodal cubic threefolds. When is there a unique IHS compactification $\overline{\mathcal{J}}(V)$ of $\mathcal{J}_{U_1}(V)$? 
\end{ques*}
More precisely, we want this compactification to preserve the fibred structure; that is we are asking when the construction of \cite{LSV} can be directly applied, so clarifying the meaning of general in their result.
\begin{thm}\label{thm:una sola LSV}
Let $V$ be a smooth cubic fourfold outside of Hassett's divisors $\mathcal{C}_8$ and $\mathcal{C}_{12}$. Then there is a unique IHS compactification of $\mathcal{J}_{U_1}(V)$. 
\end{thm}
\begin{proof}
First, let us recall that by \cite[Theorem 1.6]{Sacca:boh} there is at least one hyper-K\"ahler compactification compatible with the fibration map to $\mathbb{P}^5$.
Suppose that we have two different compactifications $\pi_i\colon\overline{\mathcal{J}}_i(V)\rightarrow \mathbb{P}^5$, for $i=1,2$.
These two manifolds are birational by construction, hence we have a natural identification 
\begin{equation}\label{eqn:two compactifications}
\oH^2(\overline{\mathcal{J}}_1(V),\ZZ)\cong \oH^2(\overline{\mathcal{J}}_2(V),\ZZ).
\end{equation}
 Notice that the indeterminacy locus of the birational map is fibred over the complement $\mathbb{P}^5\setminus U_1$. 
Let $b_i=\pi_i^*(\mathcal{O}_{\PP^5}(1))$. The variety $\mathcal{J}_{U_1}(V)$ has a distinguished relative theta divisor $T_{U_1}$ (see (\ref{eqn:Theta})), and we denote by $T_1$ and $T_2$ the closure of $T_{U_1}$ in the two compactifications, respectively. Notice that, under the identification (\ref{eqn:two compactifications}), $b_1=b_2=:b$ and $T_1=T_2$.

By construction, the divisor $b$ is nef on both manifolds. As $T_i$ is relatively ample over $U_1$, we have that $T_i+cb$ is big and nef for $c\gg0$. By hypothesis, it fails to be ample and it does not intersect some curve in the indeterminacy locus of the birational map between the two manifolds. Therefore, there is a wall divisor orthogonal to these movable classes, which is thus orthogonal to both $b$ and $T$. By Theorem~\ref{thm:ample}, this happens only when $\langle T,b\rangle^\perp$ contains a wall divisor, that is a divisor satisfying one of the following:
 \begin{enumerate}
 \item $D^2=-2$ and $\div(D)=1$,
\item $D^2=-6$ and $\div(D)=3$,
\item $D^2=-24$ and $\div(D)=3$,
\item $D^2=-4$ and $\div(D)=1$.
 \end{enumerate}
Using the isomorphism $\langle T,b\rangle^\perp\cong \oH^4(V,\ZZ)_{\operatorname{prim}}$ in Lemma~\ref{lemma:prim cohom general}, this happens only when the cubic is special in the sense of Hassett \cite{hass_cubic}, and each of these cases corresponds to one of Hassett's divisors.
By \cite{hass_cubic}, these four cases above correspond, respectively, to the four following situations:
\begin{enumerate}
\item chordal cubics (divisor $\mathcal{C}_2$);
\item nodal cubics (divisor $\mathcal{C}_6$);
\item cubics containing a plane (divisor $\mathcal{C}_8$);
\item cubics containing a cubic scroll (divisor $\mathcal{C}_{12}$).
\end{enumerate}
As the first two cases are singular cubics, our claim follows.  
\end{proof}

\subsection{Proposition \ref{wall:esclusione_a_raffica} revisited}
We conclude the paper with the following corollary of Theorem~\ref{thm:una sola LSV}.
The key argument of Proposition \ref{wall:esclusione_a_raffica} was a lattice-theoretic argument which allowed us to conclude that divisors of square $-42$ and divisibility $3$ are not wall divisors. Here we give a geometric proof of this claim.
In \cite[Section 3.2]{LSV} (see also \cite{Sacca:boh}), the authors prove that the intermediate jacobian construction works well for general Pfaffian cubic fourfolds and use this geometry to prove that their construction gives irreducible holomorphic symplectic manifolds deformation equivalent to O'Grady's tenfolds. This geometric construction is precisely what we will use.
\begin{prop}
Let $X$ be a manifold of OG10 type and let $D\subset Pic(X)$ be a divisor such that $D^2=-42$ and $div(D)=3$. Then $D$ is not a wall divisor.
\end{prop}
\begin{proof}
As divisors with these discrete properties form a single monodromy orbit, it is enough to prove our claim on a well chosen $X$. Let $V$ be a general Pfaffian cubic fourfold, and let $\IJac(V)\rightarrow \mathbb{P}^5$ be the compactification of its intermediate jacobian fibration. By \cite[Section~3.2, Theorem~4.20, Theorem~5.7]{LSV}, it is smooth. Therefore, the primitive cohomology of $V$ embeds into $\oH^2(\IJac(V),\ZZ)$ by Lemma~\ref{prop:primitive_cohom}. By work of Hassett \cite{hass_cubic}, $\oH^{2,2}(V,\ZZ)=\langle h^2,d \rangle$ where $(h^2)^2=3$, $d^2=10$ and $(h,d)=4$. It follows that $\oH^{2,2}(V,\ZZ)_{\operatorname{prim}}=\langle 3d-4h \rangle$. Put $D=3d-4h$ and, by abuse of notation, let us denote with $D$ also its image in $\Pic(\IJac(V))$ under the isomorphism in Lemma~\ref{prop:primitive_cohom}. Clearly, $D^2=-42$ and $\operatorname{div}(D)=3$. However, $D^\perp\subset \Pic(\IJac(V))$ is isometric to the hyperbolic plane $P_V$. As the isotropic class $b_V$ is nef and the relative theta divisor $T_V$ is relatively ample, it follows that $T_V+cb_V$ is big and nef for $c\gg0$. Moreover, since there exists a unique compactification for Pfaffian fourfolds (see Theorem~\ref{thm:una sola LSV}), this class must be ample. However, $D$ is orthogonal to this ample class, hence it cannot be a wall divisor.   
\end{proof}


\end{document}